\documentclass[11pt]{article}
\usepackage{amssymb}

\usepackage{graphicx}
\usepackage{amsmath}
\usepackage{makeidx}
\usepackage{indentfirst}

\newcounter{resultnum}[section]
\newtheorem{conclusion}{Conclusion}[section]

\newcounter{conclusionnum}[section]

\newcounter{conditionnum}[section]

\newcounter{conjecturenum}[section]

\newcounter{examplenum}[section]

\newcounter{exercisenum}[section]
\newtheorem{lemma}{Lemma}[section]

\newcounter{lemmanum}[section]

\newcounter{notationnum}[section]
\newtheorem{theorem}{Theorem}[section]

\newcounter{theoremnum}[section]
\newtheorem{definition}{Definition}[section]

\newcounter{definitionnum}[section]
\newtheorem{corollary}{Corollary}[section]

\newcounter{corollarynum}[section]

\newcounter{remarknum}[section]
\newtheorem{proposition}{Proposition}[section]

\newcounter{propositionnum}[section]

\newcounter{acknowledgementnum}[section]

\newcounter{algorithmnum}[section]

\newcounter{axiomnum}[section]

\newcounter{casenum}[section]
\newtheorem{claim}{Claim}[section]

\newcounter{claimnum}[section]

\newcounter{summarynum}[section]

\newcounter{problemnum}[section]
\newenvironment{proof}[1][]{\textbf{Proof.} }{}

\begin{document}

\title{Fractional Nonholonomic Ricci Flows}
\date{April 5, 2010}
\author{\textbf{Sergiu I. Vacaru} \thanks{%
sergiu.vacaru@uaic.ro, Sergiu.Vacaru@gmail.com} \\
%EndAName
\textsl{\small University "Al. I. Cuza" Ia\c si, Science Department,} \\
\textsl{\small 54 Lascar Catargi street, Ia\c si, Romania, 700107 }}
\maketitle

\begin{abstract}
We formulate the fractional Ricci flow theory for (pseudo) Riemannian geometries enabled with nonholonomic distributions defining fractional integro--differential structures, for non--integer dimensions. There are constructed fractional analogs of Perelman's functionals and derived the
corresponding fractional evolution (Hamilton's) equations. We apply in fractional calculus the nonlinear connection formalism originally elaborated in Finsler geometry and generalizations and recently applied to classical and quantum gravity
theories. There are also analyzed the fractional operators for the entropy and fundamental thermodynamic values.

\vskip0.3cm



\end{abstract}

\tableofcontents

\section{Introduction}

The purpose of this paper is to generalize the Ricci flow theory \cite%
{ham1,gper1,gper2,gper3} (see \cite{caozhu,kleiner,rbook} for reviews of results and methods) to  fractional evolution of  geometries of non--integer dimensions. The most important
achievement of this theory was the proof of W. Thurston's Geometrization
Conjecture by Grisha Perelman \cite{gper1,gper2,gper3}. The main results on
Ricci flow evolution were proved, in the bulk, for (pseudo) Riemannian and K%
\"{a}hler geometries. We show that similar results follow for geometries
with non--integer dimensions when a fractional differential and integral
calculus is corresponding encoded into nonholonomic frame structures and
adapted geometric objects.

\subsection{Basic concepts and ideas}

In a series of works (see, for instance, \cite{vnhrf1,vnhrf2,vnhrf3} and
references therein), we proved that nonholonomic constraints on Ricci flow
evolution may transform (pseudo) Riemannian metrics and Levi--Civita
connections into (in general) nonsymmetric metrics and (for instance)
Lagrange--Finsler type linear connections\footnote{%
see reviews \cite{ijgmmp,vrflg} and monograph \cite{vsgg}, and references
therein, on modern developments and applications in modern physics of the
geometry of nonholonomic manifolds and Lagrange--Finsler spaces \cite%
{ma,bejf}}. Such geometries can be modeled by nonholonomic distributions and
frames with associated nonlinear connection (N--connection) structures on
(pseudo) Riemannian spaces and/or generalizations.\footnote{%
There are used also equivalent terms like anholonomic and non--integrable
manifolds. A nonholonomic manifold is defined by a pair $(\mathbf{V},\mathcal{N%
})$, where $\mathbf{V}$ is a manifold and $\mathcal{N}$ is a nonintegrable
distribution on $\mathbf{V.}$}

If nonholonomic distributions on a manifold $M$ contain corresponding
integro--differential relations, we can model fractional geometries (for
spaces with derivatives and integrals of non--integer order). The first
example of derivative of order $\alpha =1/2$ has been described by Leibnitz
in 1695, see historical remarks in \cite{ross}. The theory of fractional
calculus with derivatives and integrals of non--integer order goes back to
Leibniz, Louville, Grunwald, Letnikov and Riemann \cite%
{oldham,podlubny,kilbas,nishimoto}. Derivatives and integrals of fractional
order, and fractional integro--differential equations, have found many
applications in physics (for example, see monographs \cite{hilfer,west} and
papers \cite{zaslavsky,taras1,taras2,baleanu}).

We consider that the question if analogous of Thurston (in particular,
Poincar\'{e}) Conjecture can be formulated (and may be proven ?) for some
spaces with fractional dimension is of fundamental importance in modern
mathematics and physics. As a first step, the goal of this paper is to
formulate a fractional version of the Hamilton--Perelman theory of Ricci
flows. On Perelman's functionals, we shall follow the methods elaborated in
Sections 1-5 of Ref. \cite{gper1} but generalized for fractional
nonholonomic manifolds by developing certain constructions from our works on
nonholonomic Ricci flow evolution \cite{vnhrf1,vnhrf2}.

We define a fractional nonholonomic manifold (equivalently, space) $\overset{%
\alpha }{\mathbf{V}}$ to be given by a quadruple $(\mathbf{V},\overset{%
\alpha }{\mathcal{N}},\overset{\alpha }{\mathbf{d}},\overset{\alpha }{%
\mathbf{I}}),$ with a fractional real number $\alpha ,$\ where $\mathbf{V}$
is a "prime" manifold of integer dimension and necessary smooth class and $%
\overset{\alpha }{\mathcal{N}}$ is a nonintegrable distribution
correspondingly adapted with a fractional differential calculus $\overset{%
\alpha }{\mathbf{d}}$ and fractional calculus $\overset{\alpha }{\mathbf{I}}$
on $\mathbf{V.}$ For such fractional nonholonomic geometries (including as
particular cases, for instance, (pseudo\footnote{%
mathematicians use the term ''semi''}) Riemannian and Finsler geometries),
we shall develop the approach to encode a corresponding fractional
integro--differential calculus adapted to nonholonomic distributions.
Perhaps, the simplest way is to follow locally a fractional vector calculus
with ''combined'' Riemann--Liouville and Caputo derivatives as in \cite%
{taras1} resulting in a self--consistent fractional generalization of
integral operations (and corresponding fractional Gauss's, Stokes', Green's
etc integral theorems which, in our approach, are crucially important for
constructing fractional Perelman's functionals).\footnote{%
There were formulated various approaches to fractional differential and
integral calculus; on nonholonomic manifolds, some of them can be related
via nonholonomic transforms/deformations of geometric structures. In this
paper, we elaborate a formalism with Caputo fractional derivative (which
gives zero acting on constants) adapted to nonlinear connections preserving
a number of similarities with an unified covariant calculus for (pseudo)
Riemannian manifolds and Lagrange--Finsler geometries \cite%
{ijgmmp,vrflg,vsgg}.}

The article is organized as follows:\ In section \ref{frnhm}, we provide a
brief introduction into the geometry of fractional nonholonomic manifolds.\
Grisha Perelman's functional approach to Ricci flow theory is generalized
for fractional nonholonomic manifolds in section \ref{ssfgpf}. We derive the
fractional nonholonomic evolution equations in section \ref{sfhe}. A
statistical interpretation of frac\-tional nonholonomic spaces and Ricci
flows is proposed. Formulas for fractional differential and integral
calculus are summarized in Appendix.

\subsection{Remarks on notations and proofs}

\label{ss12}

\begin{enumerate}
\item We shall elaborate for fractional nonholonomic spaces a system of
notations  unifying that for the nonholonomic manifolds
and bundles \cite{vrflg,vnhrf1,vnhrf2} and fractional integro--differential
calculus \cite{oldham,podlubny,kilbas,nishimoto,taras1} (we consider the
reader to be familiar with the results of such works). For geometric objects
spaces with nonholonomic distributions, we shall use boldface symbols like $%
\mathbf{V,N}$ etc and put an up label $\alpha $ for fractional
generalizations, for instance, of operators $\overset{\alpha }{\mathbf{d}},%
\overset{\alpha }{\mathbf{I}}.$ We do not use the symbol $D$ for fractional
partial derivatives as usually in works on fractional calculus (but we shall
write $\overset{\alpha }{\partial },$ or $\overset{\alpha }{d}$) because on
curved spaces of integer dimension  the symbols $D$ and $\mathbf{D}$ are
used for covariant derivatives. It is convenient to keep the same notations
of covariant derivatives  for fractional
curved spaces rewriting them, respectively, as $\overset{\alpha }{D},$ or $%
\overset{\alpha }{\mathbf{D}}.$ We shall put a ''fractional'' label $\alpha $
on the left, like $\ ^{\alpha }A,$ if that will result in a more compact
system of notations.

\item There will be also used ''up'' and ''low'' labels for some canonical
geometric objects/operators, for instance, for the horizontal (h) and
vertical (v) of a distinguished linear connection $\mathbf{D=(\ \ }_{h}D,%
\mathbf{\ \ }_{v}D\mathbf{)}$ and/or its fractional generalization $\overset{%
\alpha }{\mathbf{D}}\mathbf{=(\ \ }_{h}\overset{\alpha }{D},\mathbf{\ \ }%
\overset{\alpha }{_{v}D}\mathbf{).}$ Splitting of space dimension $\dim
\mathbf{V=}n+m$ is considered for a nonholonomic distribution (defining a
nonlinear connection, N--connection, structure) $\mathbf{N}:T\mathbf{V=}h%
\mathbf{V\oplus }v\mathbf{V,}$ where $\mathbf{\oplus }$ is the Whitney sum, $%
\dim (h\mathbf{V)=}2n$ and $\dim (v\mathbf{V)=}2m.$ In some important
particular cases, we can consider $\mathbf{V}$ to be a (pseudo) Riemannian
manifold enabled with a nonholonomic distribution $\mathcal{N}$ induced by $%
\mathbf{N,}$ or $V=E/TM$ for a vector/tangent bundle $\left( E,\pi ,M\right)
$ / $\left( TM,\pi ,M\right) $ on a manifold $M,$ $\dim M=n,$ $\dim E=n+m$ /
$\dim TM=2n,$ with $\pi $ being a corresponding surjective projection.
Indices of local coordinates on a point $\mathbf{u}$ on a open chart $%
\mathbf{U}$\textbf{\ }for an atlas $\{U\}$ covering $\mathbf{V}$ are split
in the form $\mathbf{u}^{\alpha }=(x^{i},y^{a}),$ (or in brief $u=(x,y)$),
where the general Greek indices $\alpha ,\beta ,\gamma ,...$ split
correspondingly into $h$--indices $i,j,k,...=1,2,...,n$ and $v$--indices $%
a,b,c,...=n+1,n+2,...,m.$ There are possible various types of transforms of
local frames and coframes,$\ \mathbf{e}_{\beta }=(e_{j},e_{b})$\ and $%
\mathbf{e}^{\beta }=(e^{j},e^{b})$ \ (in particular, coordinates with
''primed'', ''underlined'' indices etc), when, for instance, $\mathbf{e}%
_{\beta ^{\prime }}=e_{\ \beta ^{\prime }}^{\beta }(u)\mathbf{e}_{\beta }$
or $\mathbf{e}_{\beta }=e_{\ \beta }^{\underline{\beta }}(u)\mathbf{e}_{%
\underline{\beta }},\ \mathbf{e}_{\underline{\beta }}=\partial _{\underline{%
\beta }}=(\partial _{\underline{j}}=\partial /\partial x^{\underline{j}%
},\partial _{\underline{a}}=\partial /\partial y^{\underline{a}})$ for the
Einstein's summation rule on indices being accepted. \ Geometric objects on $%
\mathbf{V},$ for instance, tensors, connections etc can be adapted to a
N--connection structure and defined by symbols with coefficients on the
right, running corresponding values with respect to N--adapted bases
(preserving a chosen $h$--$v$--decomposition), for instance, $\mathbf{R}=\{%
\mathbf{R}_{\ \gamma \tau \mu }^{\beta }=\{R_{\ klm}^{j},R_{\ klm}^{b},R_{\
klc}^{b},...\}\}.$

\item Generalizing correspondingly a fractional integro--differential
calculus from \cite{taras1} to nonholonomic manifolds, we can elaborate a
formal (abstract) analogy with the ''integer'' case but with certain
modifications of the rules of local differentiation and ''mixed''
nonholonomic integral--differential rules. The ''fractional'' spaces and
geometric objects will be enabled with an up label $\alpha $ in the form $%
\overset{\alpha }{\mathbf{V}},\overset{\alpha }{\mathbf{D}},\overset{\alpha }%
{\mathbf{e}}^{\ \beta },\overset{\alpha }{\mathbf{R}}_{\ \gamma \tau \mu
}^{\ \beta }$ etc.

\item Following the abstract fractional N--adapted calculus, the proofs of
theorems became very similar to those given in N--adapted form \cite%
{vrflg,vnhrf1,vnhrf2}, which was used for a nonholonomic generalization of
the Ricci flow theory \cite{ham1,gper1,caozhu,kleiner,rbook}. In this paper,
we sketch proofs using the above mentioned formal analogy between N--adapted
fractional and integer geometric constructions. Proofs of the fractional
integro--differential theorems related to Ricci flow evolution became very
sophisticate if we do apply the formalism of nonholonomic distributions, do
not introduce nonlinear connections and do not apply certain methods from
the geometry of nonholonomic manifolds.
\end{enumerate}

\section{Nonholonomic Manifolds with Fractional Distributions}

\label{frnhm}

The fractional differential calculus for flat spaces elaborated in \cite%
{taras1,taras2}, and outlined in Appendix \ref{asubs1}, is extended for
nonholonomic manifolds. For simplicity, such fractional manifolds can be
modelled as real (pseudo) Riemannian spaces enabled with nonholonomic
distributions containing such integro--differential relations when the
fractional calculus on curved spaces with nonintegrable constraints is
elaborated in a form maximally similar to integer dimensions.

\subsection{Fractional (co) tangent bundles}

For the integer differential calculus, the tangent bundle $TM$ $\ $over a
manifold can be constructed for a given  local  differential structure with standard partial
derivatives $%
\partial _{i}$. Such an approach can be generalized to a fractional case
when instead of $\partial _{i}$ the differential structure is substituted,
for instance, by the left Caputo derivatives $\ _{\ _{1}x^{i}}\overset{%
\alpha }{\underline{\partial }}_{i}$ of type (\ref{lfcd}) for every local
coordinate $x^{i}$ on a local cart $X$ on $M.$ 

Let us review, in brief, the
definition for fractional tangent bundle $\overset{\alpha }{\underline{T}}M$
\ for $\alpha \in (0,1)$ (the symbol $T$ is underlined in order to emphasize
that we shall associate the approach to a fractional Caputo derivative).
Here we cite the paper \cite{albu} for some similar constructions with
fractional tangent spaces but for the left fractional RL derivative. We do
not follow that approach because it is not suitable for elaborating
fractional Ricci flow and gravitational models with exactly integrable
evolution and, respectively, field equations, and a self--consistent
fractional integral calculus with "simplified" integral theorems. For our
purposes, it is more convenient to use the fractional calculus formalism
proposed in Ref. \cite{taras1}.

We have a fractional Caputo left contact $\alpha $ in a point $\ _{0}x\in X$
\ for the parametrized curves on $M$ parametrized by a real parameter $\tau $
and $\ _{1}c,\ _{2}c:I\rightarrow M,$ with $0\in I;\ _{1}c(0)=\ _{2}c(0)\in
M $ if $\ _{\ _{1}x^{i}}\overset{\alpha }{\underline{\partial }}_{i}(f\circ
\ _{1}c)_{\mid \tau =0}=\ _{\ _{1}x^{i}}\overset{\alpha }{\underline{%
\partial }}_{i}(f\circ \ _{2}c)_{\mid \tau =0}$ holds for all analytic
functions $f$ on $X.$ This defines a relation of equivalence when the
classes $[\overset{\alpha }{\underline{c}}]_{\ _{0}x}$ determines the
fractional left tangent Caputo space $\overset{\alpha }{\underline{T}}_{\
_{0}x}M.$ The corresponding fractional tangent bundle is $\overset{\alpha }{%
\underline{T}}M:=\bigcup\limits_{_{0}x}$ $\overset{\alpha }{\underline{T}}%
_{\ _{0}x}M$ $\ $when the surjective projection $\ \overset{\alpha }{\pi }:%
\overset{\alpha }{\underline{T}}M\rightarrow M$ \ acts as $\overset{\alpha }{%
\pi }[\overset{\alpha }{\underline{c}}]_{\ _{0}x}=\ _{0}x.$ Such a
fractional bundle space is given by a triple $\left( \overset{\alpha }{%
\underline{T}}M,\overset{\alpha }{\pi },M\right) .$ For simplicity, we shall
write, in brief, for the total space only the symbol  $\overset{\alpha }{\underline{T}}M $ \
if that will not result in ambiguities.

Locally on $M,$ the class $[\overset{\alpha }{\underline{c}}]_{\ _{0}x}$ is
characterized by a curve
\begin{equation*}
x^{i}(\tau )=x^{i}(0)+\frac{\tau ^{\alpha }}{\Gamma (1+\alpha )}_{\ _{1}c}%
\overset{\alpha }{\underline{\partial }}_{\tau }x_{\mid \tau =0}^{i},
\end{equation*}%
for $\tau \in (-\varepsilon ,\varepsilon ).$ So, the horizontal and vertical
coordinates (respectively, h-- and v--coordinates)\ \ on $\overset{\alpha }{%
\pi }^{-1}(X)\subset \overset{\alpha }{\underline{T}}_{\ _{0}x}M$ are $\
\overset{\alpha }{u}^{\beta }=(x^{j},\overset{\alpha }{y}^{j}),$ where
\begin{equation*}
x^{i}=x^{i}(0)\mbox{\ and \ }\overset{\alpha }{y}^{j}=\frac{1}{\Gamma
(1+\alpha )}_{\ _{1}c}\overset{\alpha }{\underline{\partial }}_{\tau
}x^{i}(\tau )_{\mid \tau =0}.
\end{equation*}%
For simplicity, we shall write instead of$\ \overset{\alpha }{u}^{\beta
}=(x^{j},\overset{\alpha }{y}^{j}),$ the local coordinates $u^{\beta
}=(x^{j},y^{j})$ both for integer and fractional tangent bundles considering
that there were chosen certain such parametrizations of local coordinate
systems by using classes of equivalence only for the left Caputo fractional
derivatives.

On $\overset{\alpha }{\underline{T}}M,$ we can consider an arbitrary
fractional left (Caputo) frame basis
\begin{equation}
\overset{\alpha }{\underline{e}}_{\beta }=e_{\ \beta }^{\beta ^{\prime
}}(u^{\beta })\overset{\alpha }{\underline{\partial }}_{\beta ^{\prime }}
\label{flcfb}
\end{equation}%
where the fractional local coordinate basis
\begin{equation}
\overset{\alpha }{\underline{\partial }}_{\beta ^{\prime }}=\left( \overset{%
\alpha }{\underline{\partial }}_{j^{\prime }}=_{\ _{1}x^{^{\prime }}}\overset%
{\alpha }{\underline{\partial }}_{j},\overset{\alpha }{\underline{\partial }}%
_{b^{\prime }}=_{\ _{1}y^{b^{\prime }}}\overset{\alpha }{\underline{\partial
}}_{b^{\prime }}\right)  \label{frlcb}
\end{equation}%
is with running of indices of type $j^{\prime }=1,2,...,n$ and $b^{\prime
}=n+1,n+2,...,n+n.$ We might introduce arbitrary fractional co--bases which
are dual to (\ref{flcfb}),
\begin{equation}
\overset{\alpha }{\underline{e}}^{\ \beta }=e_{\beta ^{\prime }\ }^{\ \beta
}(u^{\beta })\overset{\alpha }{d}u^{\beta ^{\prime }},  \label{flcfdb}
\end{equation}%
where the fractional local coordinate co--basis
\begin{equation}
\ _{\ }\overset{\alpha }{d}u^{\beta ^{\prime }}=\left( (dx^{i^{\prime
}})^{\alpha },(dy^{a^{\prime }})^{\alpha }\right)  \label{frlccb}
\end{equation}%
with h-- and v--components, $(dx^{i^{\prime }})^{\alpha }$ and $%
(dy^{a^{\prime }})^{\alpha }$ being of type (\ref{fr1f}). For integer
values, a matrix $e_{\ \beta }^{\beta ^{\prime }}$ is inverse to $e_{\beta
^{\prime }\ }^{\ \beta }$ (physicists call such matrices as vielbeins). Such
a property holds for fractional constructions with the left Caputo
fractional derivatives.

Similarly to $\overset{\alpha }{\underline{T}}M,$ the above constructions
can be used for definition of fractional vector bundle $\overset{\alpha }{%
\underline{E}}$ on $M,$ when the fiber indices of bases run values $%
a^{\prime },b^{\prime },...=n+1,n+2,...,n+m.$

For a tangent bundle $TM,$ using 1--forms as respective duals local vector
bases, we can define the co--tangent bundle $T^{\ast }M$ on $M.$ Using
corresponding definitions of fractional forms (\ref{fr1f}) and dubbing the
for fractional differentials the above constructions, we can construct the
fractional cotangent bundle $\overset{\alpha }{\underline{T}}^{\ \ast }M$
and consider generalizations for co--vector bundle $\overset{\alpha \ }{%
\underline{E}}^{\ast }$ on $M.$ We omit details on such constructions (and
possible higher order fractional tangent/vector generalizations, fractional
osculator bundles etc) in this paper.

\subsection{Fundamental geometric objects on fractional manifolds}

Let us consider a fractional nonholonomic manifold $\overset{\alpha }{%
\mathbf{V}}$ defined by a quadruple $(\mathbf{V},\overset{\alpha }{\mathcal{N%
}},\overset{\alpha }{\mathbf{d}},\overset{\alpha }{\mathbf{I}}),$ where the
fractional differential structure $\overset{\alpha }{\mathbf{d}}$ is stated
by some (\ref{flcfb}) and (\ref{flcfdb}) and the non--integer integral
structure $\overset{\alpha }{\mathbf{I}}$ is given by rules of type (\ref%
{aux01}). A ''prime'' integer manifold $\mathbf{V}$ is of integer dimension $%
\dim $ $\mathbf{V}=n+m,n\geq 2,m\geq 1.$\ Local coordinates on $\mathbf{V}$
are labeled in the form $u=(x,y),$ or $u^{\alpha }=(x^{i},y^{a}),$ where
indices $i,j,...=1,2,...,n$ are horizontal (h) ones and $a,b,...$ $%
=1,2,...,m $ are vertical (v) ones. For some important examples, we have
that $\mathbf{V}=TM$ is a tangent bundle, or $\mathbf{V=E}$ is a vector
bundle, on $M,$ or $\mathbf{V}$ is a (semi--) Riemann manifold, with
prescribed local (non--integrable) fibred structure. A nonholonomic manifold
$\mathbf{V} $ is considered to be enabled with a non---integrable
distribution defining a nonlinear connection as we explained in point 2 of
section \ref{ss12}.

\subsubsection{N--connections for fractional nonholonomic manifolds}

A nonintegrable distribution $\overset{\alpha }{\mathcal{N}}$ $\ $for $%
\overset{\alpha }{\mathbf{V}}$ can be chosen in a form defining a nonlinear
connection structure correspondingly adapted to chosen fractional calculus
with $\overset{\alpha }{\mathbf{d}}$ and $\overset{\alpha }{\mathbf{I}}.$

\begin{definition}
A nonlinear connection (N--connection) $\overset{\alpha }{\mathbf{N}}$ is
defined by a Whitney sum of conventional h-- and v--subspaces, $\underline{h}%
\overset{\alpha }{\mathbf{V}}$ and $\underline{v}\overset{\alpha }{\mathbf{V}%
},$%
\begin{equation}
\overset{\alpha }{\underline{T}}\overset{\alpha }{\mathbf{V}}=\underline{h}%
\overset{\alpha }{\mathbf{V}}\mathbf{\oplus }\underline{v}\overset{\alpha }{%
\mathbf{V}}\mathbf{,}  \label{whit}
\end{equation}%
where the fractional tangent bundle $\overset{\alpha }{\underline{T}}\overset%
{\alpha }{\mathbf{V}}$ is constructed following the approach with the left
Caputo fractional derivative chosen for the differential structure.
\end{definition}

We note that a conventional splitting into $h$- and $v$-components depends
on the type of chosen fractional left derivative. We underline some symbols
if it is important to emphasize that the corresponding geometric objects are
induced by the Caputo fractional derivative, but we shall omit "underling"
if that will simplify the system of notation not resulting in ambiguities.

Nonholonomic manifolds with $\overset{\alpha }{\mathcal{N}}$ \ determined by
a $\overset{\alpha }{\mathbf{N}}$ are called N--anholonomic
fractional manifolds. In brief, we shall call them as fractional spaces
(geometries/manifolds). Locally, a fractional N--connection is defined by
its coefficients, $\overset{\alpha }{\mathbf{N}}\mathbf{=}\{\ ^{\alpha
}N_{i}^{a}\},$ stated with respect to a local coordinate basis,
\begin{equation}
\overset{\alpha }{\mathbf{N}}\mathbf{=}\ ^{\alpha
}N_{i}^{a}(u)(dx^{i})^{\alpha }\otimes \overset{\alpha }{\underline{\partial
}}_{a},  \label{fnccoef}
\end{equation}%
see formulas (\ref{frlcb}) and (\ref{frlccb}).

N--connections are naturally considered in Finsler and Lagrange geometry,
Einstein gravity, and various supersymmetric, noncommutative, quantum
generalizations in modern (super) string/brane theories and geometric
mechanics, see reviews of results in \cite%
{ma,bejf,ijgmmp,vrflg,vsgg,vstring,vncg,vbrane}.

\begin{proposition}
A N--connection $\overset{\alpha }{\mathbf{N}}$ defines N--adapted (i.e.
linearly depending on coefficients $\ ^{\alpha }N_{i}^{a})$ fractional frame
\begin{equation}
\ ^{\alpha }\mathbf{e}_{\beta }=\left[ \ ^{\alpha }\mathbf{e}_{j}=\overset{%
\alpha }{\underline{\partial }}_{j}-\ ^{\alpha }N_{j}^{a}\overset{\alpha }{%
\underline{\partial }}_{a},\ ^{\alpha }e_{b}=\overset{\alpha }{\underline{%
\partial }}_{b}\right]  \label{dder}
\end{equation}%
and coframe
\begin{equation}
\ ^{\alpha }\mathbf{e}^{\beta }=[\ ^{\alpha }e^{j}=(dx^{j})^{\alpha },\
^{\alpha }\mathbf{e}^{b}=(dy^{b})^{\alpha }+\ ^{\alpha
}N_{k}^{b}(dx^{k})^{\alpha }]  \label{ddif}
\end{equation}%
nonholonomic structures.
\end{proposition}

\begin{proof}
The corresponding nonholonomic integro--differential fractional structure is
induced by the left Caputo derivative (\ref{lfcd}) and N--connection
coefficients in (\ref{fnccoef}). The nontrivial nonholonomy coefficients are
computed $\ ^{\alpha }W_{ib}^{a}=\overset{\alpha }{\underline{\partial }}%
_{b}\ ^{\alpha }N_{i}^{a}$ and $\ ^{\alpha }W_{ij}^{a}=\ ^{\alpha }\Omega
_{ji}^{a}=\ ^{\alpha }\mathbf{e}_{i}\ ^{\alpha }N_{j}^{a}-\ ^{\alpha }%
\mathbf{e}_{j}\ ^{\alpha }N_{i}^{a}$ (where $\ ^{\alpha }\Omega _{ji}^{a}$
are the coefficients of the N--connection curvature) for
\begin{equation*}
\left[ \ ^{\alpha }\mathbf{e}_{\alpha },\ ^{\alpha }\mathbf{e}_{\beta }%
\right] =\ ^{\alpha }\mathbf{e}_{\alpha }\ ^{\alpha }\mathbf{e}_{\beta }-\
^{\alpha }\mathbf{e}_{\beta }\ ^{\alpha }\mathbf{e}_{\alpha }=\ ^{\alpha
}W_{\alpha \beta }^{\gamma }\ ^{\alpha }\mathbf{e}_{\gamma }.
\end{equation*}%
For simplicity, in above formulas derived for (\ref{dder}) and (\ref{ddif}),
we omitted underlying of symbols of type $\ ^{\alpha }\mathbf{e}_{\beta }=[\
^{\alpha }\mathbf{e}_{j},\ ^{\alpha }e_{b}]$ even such values are determined
by fractional Caputo derivatives of type (\ref{lfcd}), which are underlined.

(End proof.)\ $\square $
\end{proof}

\subsubsection{N--adapted fractional metrics}

A second fundamental geometric object on $\overset{\alpha }{\mathbf{V}},$ a
metric $\overset{\alpha }{\mathbf{g}},$ can be defined similarly to (pseudo)
Riemannian spaces of integer dimension but for a chosen fractional
differential structure.

\begin{definition}
A (fractional) metric structure $\overset{\alpha }{\mathbf{g}}=\{\ ^{\alpha
}g_{\underline{\alpha }\underline{\beta }}\}$ is determined on a $\overset{%
\alpha }{\mathbf{V}}$ \ by a symmetric second rank tensor
\begin{equation}
\overset{\alpha }{\mathbf{g}}=\ ^{\alpha }g_{\underline{\gamma }\underline{%
\beta }}(u)(du^{\underline{\gamma }})^{\alpha }\otimes (du^{\underline{\beta
}})^{\alpha }  \label{fmcf}
\end{equation}%
for a tensor product of fractional coordinate co--bases (\ref{frlccb}).
\end{definition}

For N--adapted constructions, it is important to introduce and prove:

\begin{claim}
Any fractional metric $\overset{\alpha }{\mathbf{g}}$ can be represented
equivalently as a distinguished metric structure (d--metric), $\ \overset{%
\alpha }{\mathbf{g}}=\left[ \ ^{\alpha }g_{kj},\ ^{\alpha }g_{cb}\right] ,$
which is N--adapted to splitting (\ref{whit}),
\begin{equation}
\ \overset{\alpha }{\mathbf{g}}=\ ^{\alpha }g_{kj}(x,y)\ ^{\alpha
}e^{k}\otimes \ ^{\alpha }e^{j}+\ ^{\alpha }g_{cb}(x,y)\ ^{\alpha }\mathbf{e}%
^{c}\otimes \ ^{\alpha }\mathbf{e}^{b},  \label{m1}
\end{equation}%
where fractional N--elongated bases $\ ^{\alpha }\mathbf{e}^{\beta }=[\
^{\alpha }e^{j},\ ^{\alpha }\mathbf{e}^{b}]$ \ are defined as in (\ref{dder}%
).
\end{claim}

\begin{proof}
For coefficients of metric (\ref{fmcf}), we consider parametrization
\begin{equation}
\ \ ^{\alpha }g_{\underline{\alpha }\underline{\beta }}=\left[
\begin{array}{cc}
\ ^{\alpha }\underline{g}_{ij}=\ \ ^{\alpha }g_{ij}+~\ ^{\alpha }N_{i}^{a}~\
^{\alpha }N_{j}^{b}\ \ ^{\alpha }g_{ab} & ~\ ^{\alpha }\underline{g}_{ib}=~\
^{\alpha }N_{i}^{e}\ \ ^{\alpha }g_{be} \\
~\ ^{\alpha }\underline{g}_{aj}=\ ^{\alpha }N_{i}^{e}\ \ ^{\alpha }g_{be} &
\ \ ^{\alpha }\underline{g}_{ab}%
\end{array}%
\right] ,  \label{qel}
\end{equation}%
for $\ ^{\alpha }g_{\underline{\alpha }\underline{\beta }}=\ ^{\alpha }%
\underline{g}_{\alpha \beta }.$ We introduce the vielbeins
\begin{equation}
\mathbf{e}_{\alpha }^{\ \underline{\alpha }}=\left[
\begin{array}{cc}
e_{i}^{\ \underline{i}}=\delta _{i}^{\underline{i}} & e_{i}^{\ \underline{a}%
}=\ ^{\alpha }N_{i}^{b}\ \delta _{b}^{\underline{a}} \\
e_{a}^{\ \underline{i}}=0 & e_{a}^{\ \underline{a}}=\delta _{a}^{\underline{a%
}}%
\end{array}%
\right] ,\ \mathbf{e}_{\ \underline{\alpha }}^{\alpha }=\left[
\begin{array}{cc}
e_{\ \underline{i}}^{i}=\delta _{\underline{i}}^{i} & e_{\ \underline{i}%
}^{b}=-\ ^{\alpha }N_{k}^{b}\ \delta _{\underline{i}}^{k} \\
e_{\ \underline{a}}^{i}=0 & e_{\ \underline{a}}^{a}=\delta _{\underline{a}%
}^{a}%
\end{array}%
\right] ,  \label{fr}
\end{equation}%
where $\delta _{\underline{i}}^{i}$ is the Kronecher symbol, and define
nonholonomic frames
\begin{equation*}
\ ^{\alpha }\mathbf{e}_{\beta }=\mathbf{e}_{\beta }^{\ \underline{\beta }}\
\partial _{\underline{\beta }}\mbox{\ and \ }\ ^{\alpha }\mathbf{e}^{\alpha
}=\mathbf{e}_{\ \underline{\beta }}^{\beta }(du^{\underline{\beta }%
})^{\alpha },
\end{equation*}%
which are N--adapted frames, respectively, of type (\ref{dder}) and (\ref%
{ddif}). Re--grouping the coefficients, we get the formula (\ref{m1}). $%
\square $
\end{proof}

\subsubsection{Distinguished fractional connections}

Linear connections on fractional $\overset{\alpha }{\mathbf{V}}$ may be
adapted to the N--connection structure as for the integer dimensions.

\begin{definition}
A distinguished connection (d--connection) $\overset{\alpha }{\mathbf{D}}$
on $\overset{\alpha }{\mathbf{V}}$ is a linear connection preserving under
parallel transports the Whitney sum (\ref{whit}).
\end{definition}

A covariant fractional calculus on nonholonomic manifolds can be developed
following the formalism of fractional differential forms. For a fractional
d--connection $\overset{\alpha }{\mathbf{D}},$ \ we can introduce a
N--adapted differential 1--form of type (\ref{fr1f})
\begin{equation}
\ ^{\alpha }\mathbf{\Gamma }_{\ \beta }^{\tau }=\ ^{\alpha }\mathbf{\Gamma }%
_{\ \beta \gamma }^{\tau }\ ^{\alpha }\mathbf{e}^{\gamma },  \label{fdcf}
\end{equation}%
with the coefficients defined with respect to (\ref{ddif}) and (\ref{dder})
and parametrized  the form $\ ^{\alpha }\mathbf{\Gamma }_{\ \tau \beta
}^{\gamma }=\left( \ ^{\alpha }L_{jk}^{i},\ ^{\alpha }L_{bk}^{a},\ ^{\alpha
}C_{jc}^{i},\ ^{\alpha }C_{bc}^{a}\right) .$

We also consider that the absolute fractional differential $\ ^{\alpha }%
\mathbf{d}=\ _{\ _{1}x}\overset{\alpha }{d}_{x}+\ _{\ _{1}y}\overset{\alpha }%
{d}_{y}$ is a N--adapted fractional operator $\ ^{\alpha }\mathbf{d:=}\
^{\alpha }\mathbf{e}^{\beta }\ ^{\alpha }\mathbf{e}_{\beta }$ defined by
exterior h- and v--derivatives of type (\ref{feder}), \ when \
\begin{equation*}
\ _{\ _{1}x}\overset{\alpha }{d}_{x}:=(dx^{i})^{\alpha }\ \ _{\ _{1}x}%
\overset{\alpha }{\underline{\partial }}_{i}=\ ^{\alpha }e^{j}\ ^{\alpha }%
\mathbf{e}_{j}\mbox{ and }_{\ _{1}y}\overset{\alpha }{d}_{y}:=(dy^{a})^{%
\alpha }\ \ _{\ _{1}x}\overset{\alpha }{\underline{\partial }}_{a}=\
^{\alpha }\mathbf{e}^{b}\ ^{\alpha }e_{b}.
\end{equation*}

\begin{definition}
The torsion of a fractional d--connection $\overset{\alpha }{\mathbf{D}}=\{\
^{\alpha }\mathbf{\Gamma }_{\ \beta \gamma }^{\tau }\}$ is
\begin{equation}
\ ^{\alpha }\mathcal{T}^{\tau }\doteqdot \overset{\alpha }{\mathbf{D}}\
^{\alpha }\mathbf{e}^{\tau }=\ ^{\alpha }\mathbf{d}\ ^{\alpha }\mathbf{e}%
^{\tau }+\ ^{\alpha }\mathbf{\Gamma }_{\ \beta }^{\tau }\wedge \ ^{\alpha }%
\mathbf{e}^{\beta }.  \label{tors}
\end{equation}
\end{definition}

Following an explicit fractional (and N--adapted) differential form calculus
with respect to (\ref{ddif}), we prove:

\begin{theorem}
Locally, the fractional torsion $\ ^{\alpha }\mathcal{T}^{\tau }$ (\ref{tors}%
) is characterized by its coefficients (d--torsion)
\begin{eqnarray}
\ ^{\alpha }T_{\ jk}^{i} &=&\ ^{\alpha }L_{\ jk}^{i}-\ ^{\alpha }L_{\
kj}^{i},\ \ ^{\alpha }T_{\ ja}^{i}=-\ ^{\alpha }T_{\ aj}^{i}=\ ^{\alpha
}C_{\ ja}^{i},\ \ ^{\alpha }T_{\ ji}^{a}=\ ^{\alpha }\Omega _{\ ji}^{a},\
\notag \\
\ ^{\alpha }T_{\ bi}^{a} &=&-\ ^{\alpha }T_{\ ib}^{a}=\ ^{\alpha }e_{b}\
^{\alpha }N_{i}^{a}-\ ^{\alpha }L_{\ bi}^{a},\ \ ^{\alpha }T_{\ bc}^{a}=\
^{\alpha }C_{\ bc}^{a}-\ ^{\alpha }C_{\ cb}^{a}.  \label{dtors}
\end{eqnarray}
\end{theorem}

For integer $\alpha ,$ we get the same formulas as in \cite%
{ijgmmp,vrflg,vsgg,ma}. This is possible if we consider on $\overset{\alpha }%
{\mathbf{V}}$ a differential structure which locally can be induced by the
left Caputo fractional derivatives and associated differentials.

\begin{definition}
The curvature of a fractional $\overset{\alpha }{\mathbf{D}}=\{\ ^{\alpha }%
\mathbf{\Gamma }_{\ \beta \gamma }^{\tau }\}$ is
\begin{equation}
\ ^{\alpha }\mathcal{R}_{~\beta }^{\tau }\doteqdot \overset{\alpha }{\mathbf{%
D}}\mathbf{\ ^{\alpha }\Gamma }_{\ \beta }^{\tau }=\ ^{\alpha }\mathbf{d\
^{\alpha }\Gamma }_{\ \beta }^{\tau }-\ ^{\alpha }\mathbf{\Gamma }_{\ \beta
}^{\gamma }\wedge \ ^{\alpha }\mathbf{\Gamma }_{\ \gamma }^{\tau }=\
^{\alpha }\mathbf{R}_{\ \beta \gamma \delta }^{\tau }\ ^{\alpha }\mathbf{e}%
^{\gamma }\wedge \ ^{\alpha }\mathbf{e}  \label{dcv}
\end{equation}
\end{definition}

A straightforward fractional differential form calculus for (\ref{fdcf})
gives a proof of

\begin{theorem}
Locally, the fractional curvature $\ ^{\alpha }\mathcal{R}_{~\beta }^{\tau }$
(\ref{dcv}) is characterized by its coefficients (d--curvature)
\begin{eqnarray}
\ ^{\alpha }R_{\ hjk}^{i} &=&\ ^{\alpha }\mathbf{e}_{k}\ ^{\alpha }L_{\
hj}^{i}-\ ^{\alpha }\mathbf{e}_{j}\ ^{\alpha }L_{\ hk}^{i}  \notag \\
&&+\ ^{\alpha }L_{\ hj}^{m}\ ^{\alpha }L_{\ mk}^{i}-\ ^{\alpha }L_{\
hk}^{m}\ ^{\alpha }L_{\ mj}^{i}-\ ^{\alpha }C_{\ ha}^{i}\ ^{\alpha }\Omega
_{\ kj}^{a},  \notag \\
\ ^{\alpha }R_{\ bjk}^{a} &=&\ ^{\alpha }\mathbf{e}_{k}\ ^{\alpha }L_{\
bj}^{a}-\ ^{\alpha }\mathbf{e}_{j}\ ^{\alpha }L_{\ bk}^{a}  \notag \\
&&+\ ^{\alpha }L_{\ bj}^{c}\ ^{\alpha }L_{\ ck}^{a}-\ ^{\alpha }L_{\
bk}^{c}\ ^{\alpha }L_{\ cj}^{a}-\ ^{\alpha }C_{\ bc}^{a}\ ^{\alpha }\Omega
_{\ kj}^{c},  \notag \\
\ ^{\alpha }R_{\ jka}^{i} &=&\ ^{\alpha }e_{a}\ ^{\alpha }L_{\ jk}^{i}-\
^{\alpha }D_{k}\ ^{\alpha }C_{\ ja}^{i}+\ ^{\alpha }C_{\ jb}^{i}T_{\ ka}^{b},
\label{dcurv} \\
\ ^{\alpha }R_{\ bka}^{c} &=&\ ^{\alpha }e_{a}\ ^{\alpha }L_{\ bk}^{c}-\
^{\alpha }D_{k}\ ^{\alpha }C_{\ ba}^{c}+\ ^{\alpha }C_{\ bd}^{c}\ ^{\alpha
}T_{\ ka}^{c},  \notag \\
\ ^{\alpha }R_{\ jbc}^{i} &=&\ ^{\alpha }e_{c}\ ^{\alpha }C_{\ jb}^{i}-\
^{\alpha }e_{b}\ ^{\alpha }C_{\ jc}^{i}+\ ^{\alpha }C_{\ jb}^{h}\ ^{\alpha
}C_{\ hc}^{i}-\ ^{\alpha }C_{\ jc}^{h}\ ^{\alpha }C_{\ hb}^{i},  \notag \\
\ ^{\alpha }R_{\ bcd}^{a} &=&\ ^{\alpha }e_{d}\ ^{\alpha }C_{\ bc}^{a}-\
^{\alpha }e_{c}\ ^{\alpha }C_{\ bd}^{a}+\ ^{\alpha }C_{\ bc}^{e}\ ^{\alpha
}C_{\ ed}^{a}-\ ^{\alpha }C_{\ bd}^{e}\ ^{\alpha }C_{\ ec}^{a}.  \notag
\end{eqnarray}
\end{theorem}

Formulas (\ref{dtors}) and (\ref{dcurv}) encode integro--differential
nonholonomic structures modeling certain types of fractional differential
geometric models. For integer dimensions, on vector/tangent bundles, such
constructions are typical ones for Lagrange--Finsler geometry \cite{ma} and
various types generalizations in modern geometry and gravity \cite%
{vncg,vsgg,vrflg,vbrane}.

Contracting respectively the components of (\ref{dcurv}), we can prove

\begin{proposition}
The fractional Ricci tensor $\ ^{\alpha }\mathcal{R}ic=\{\ ^{\alpha }\mathbf{%
R}_{\alpha \beta }\doteqdot \ ^{\alpha }\mathbf{R}_{\ \alpha \beta \tau
}^{\tau }\}$ is characterized by h- v--components, i.e. d--tensors,%
\begin{equation}
\ ^{\alpha }R_{ij}\doteqdot \ ^{\alpha }R_{\ ijk}^{k},\ \ \ ^{\alpha
}R_{ia}\doteqdot -\ ^{\alpha }R_{\ ika}^{k},\ \ ^{\alpha }R_{ai}\doteqdot \
^{\alpha }R_{\ aib}^{b},\ \ ^{\alpha }R_{ab}\doteqdot \ ^{\alpha }R_{\
abc}^{c}.  \label{dricci}
\end{equation}
\end{proposition}

It is obvious that the fractional Ricci tensor $\ ^{\alpha }\mathbf{R}%
_{\alpha \beta }$ is not symmetric for arbitrary fractional d--connecti\-ons.

For a fractional d--metric structure (\ref{m1}), we can introduce:

\begin{definition}
The scalar curvature of a fractional d--connection $\overset{\alpha }{%
\mathbf{D}}$ is
\begin{eqnarray}
\ _{s}^{\alpha }\mathbf{R} &\doteqdot &\ ^{\alpha }\mathbf{g}^{\tau \beta }\
^{\alpha }\mathbf{R}_{\tau \beta }=\ ^{\alpha }R+\ ^{\alpha }S,
\label{sdccurv} \\
\ ^{\alpha }R &=&\ ^{\alpha }g^{ij}\ ^{\alpha }R_{ij},\ \ ^{\alpha }S=\
^{\alpha }g^{ab}\ ^{\alpha }R_{ab},  \notag
\end{eqnarray}%
defined by a sum the h-- and v--components of (\ref{dricci}) and
contractions with the inverse coefficients to a d--metric (\ref{m1}).
\end{definition}

\begin{proposition}
\textbf{-Definition: } The Einstein tensor $\ ^{\alpha }\mathcal{E}ns=\{\
^{\alpha }\mathbf{G}_{\alpha \beta }\}$ for a fractional d--connection $%
\overset{\alpha }{\mathbf{D}}$ is computed in standard form
\begin{equation}
\ ^{\alpha }\mathbf{G}_{\alpha \beta }:=\ ^{\alpha }\mathbf{R}_{\alpha \beta
}-\frac{1}{2}\ ^{\alpha }\mathbf{g}_{\alpha \beta }\ \ _{s}^{\alpha }\mathbf{%
R.}  \label{enstdt}
\end{equation}
\end{proposition}

Such a tensor can be used for various fractional generalizations of the
Einstein and Lagrange--Finsler gravity models from \cite%
{ijgmmp,vrflg,vsgg,ma}. It should be emphasized that variants of fractional
Ricci and Einstein tensor were considered in \cite{albu,munk}, respectively,
for generalized fractional Rimann--Finsler and Einstein spaces but with RL
fractional derivatives. Technically, it is a very cumbersome task to find
solutions of such sophisticate integro--differential equations and study
possible physical implications. In our approach, working with the left
Caputo fractional derivative and by corresponding nonholonomic transforms,
we can separate the equations in fractional equations in such a form that
the resulting systems of partial differential and integral equations can
integrated exactly in very general form similarly to the integer cases
outlined for different models of gravity theory in \cite%
{vexsol,ijgmmp,vrflg,vsgg} and, for nonholonomic Ricci flows and
applications to physics, in \cite{nhrfs1,nhrfs2,nhrfs2,rfv1,rfvv,vnhrf3}.

\subsubsection{The fractional canonical d--connection and Levi--Civita
connection}

There are an infinite number of fractional d--connections $\overset{\alpha }{%
\mathbf{D}}$ on $\overset{\alpha }{\mathbf{V}}.$ For applications in modern
geometry and physics, a special interest present subclasses of such linear
connections which are metric compatible with a metric structure, i.e. $%
\overset{\alpha }{\mathbf{D}}\left( \ ^{\alpha }\mathbf{g}\right) =0,$ with
more special cases when $\overset{\alpha }{\mathbf{D}}$ is completely and
uniquely determined by $\ ^{\alpha }\mathbf{g}$ and $\overset{\alpha }{%
\mathbf{N}}$\textbf{\ }following certain well--defined geometric/physical
principles.

\begin{theorem}
There is a unique canonical fractional d--connection $\ ^{\alpha }\widehat{%
\mathbf{D}}=\{\ ^{\alpha }\widehat{\mathbf{\Gamma }}_{\ \alpha \beta
}^{\gamma }=\left( \ ^{\alpha }\widehat{L}_{jk}^{i},\ ^{\alpha }\widehat{L}%
_{bk}^{a},\ ^{\alpha }\widehat{C}_{jc}^{i},\ ^{\alpha }\widehat{C}%
_{bc}^{a}\right) \}$ which is compatible with the metric structure, $\
^{\alpha }\widehat{\mathbf{D}}\ \left( \ ^{\alpha }\mathbf{g}\right) =0,$
and satisfies the conditions $\ ^{\alpha }\widehat{T}_{\ jk}^{i}=0$ and $\
^{\alpha }\widehat{T}_{\ bc}^{a}=0.$
\end{theorem}

\begin{proof}
It follows from explicit formulas for coefficients of (\ref{m1}) and
\begin{eqnarray}
\ ^{\alpha }\widehat{L}_{jk}^{i} &=&\frac{1}{2}\ ^{\alpha }g^{ir}\left( \
^{\alpha }\mathbf{e}_{k}\ ^{\alpha }g_{jr}+\ ^{\alpha }\mathbf{e}_{j}\
^{\alpha }g_{kr}-\ ^{\alpha }\mathbf{e}_{r}\ ^{\alpha }g_{jk}\right) ,
\label{candcon} \\
\ ^{\alpha }\widehat{L}_{bk}^{a} &=&\ ^{\alpha }e_{b}(\ ^{\alpha }N_{k}^{a})+
\notag \\
&&\frac{1}{2}\ ^{\alpha }g^{ac}\left( \ ^{\alpha }\mathbf{e}_{k}\ ^{\alpha
}g_{bc}-\ ^{\alpha }g_{dc}\ \ ^{\alpha }e_{b}\ ^{\alpha }N_{k}^{d}-\
^{\alpha }g_{db}\ \ ^{\alpha }e_{c}\ ^{\alpha }N_{k}^{d}\right) ,  \notag \\
\ ^{\alpha }\widehat{C}_{jc}^{i} &=&\frac{1}{2}\ ^{\alpha }g^{ik}\ ^{\alpha
}e_{c}\ ^{\alpha }g_{jk},  \notag \\
\ \ ^{\alpha }\widehat{C}_{bc}^{a} &=&\frac{1}{2}\ ^{\alpha }g^{ad}\left( \
^{\alpha }e_{c}\ ^{\alpha }g_{bd}+\ ^{\alpha }e_{c}\ ^{\alpha }g_{cd}-\
^{\alpha }e_{d}\ ^{\alpha }g_{bc}\right) .  \notag
\end{eqnarray}%
Introducing the values (\ref{candcon}) into formulas (\ref{dtors}) we obtain
that $\widehat{T}_{\ jk}^{i}=0$ and $\widehat{T}_{\ bc}^{a}=0,$ but $%
\widehat{T}_{\ ja}^{i},\widehat{T}_{\ ji}^{a}$ and $\widehat{T}_{\ bi}^{a}$
are not zero, that the metricity conditions are satisfied in component form.
$\square $
\end{proof}

On a fractional nonholonomic $\overset{\alpha }{\mathbf{V}},$ the
Levi--Civita connection $\ ^{\alpha }\nabla =\{\ \ ^{\alpha }\Gamma _{\
\alpha \beta }^{\gamma }\}$ can be defined in standard from by using the
fractional Caputo left derivatives acting correspondingly on the
coefficients of a fractional metric (\ref{fmcf}). Such a geometric object is
not adapted to the N--connection splitting (\ref{whit}). As a consequence of
nonholonomic structure, it is preferred to work on $\overset{\alpha }{%
\mathbf{V}}$ with $\ ^{\alpha }\widehat{\mathbf{D}}=\{\ ^{\alpha }\widehat{%
\mathbf{\Gamma }}_{\ \tau \beta }^{\gamma }\}$ instead of $\ ^{\alpha
}\nabla .$ Even \ $\ ^{\alpha }\widehat{\mathbf{D}}$ has a nontrivial
d--torsion, such an object is very different from a similar one, for
instance, in "integer" Einstein--Cartan gravity when additional
gravitational equations have to be introduces for the nontrivial torsion
components. In our case, the canonical $\ ^{\alpha }\widehat{\mathcal{T}}%
^{\tau }$ (\ref{tors}) is nonholonomically induced, via fractional integral
and derivative operations, in a unique form, by some off--diagonal
coefficients of metric field.

Let us parametrize the coefficients of $\ ^{\alpha }\nabla $ (for integer $%
\alpha ,$ it is uniquely derived from the conditions $~\ _{\shortmid }%
\mathcal{T}=0$ and $\bigtriangledown g=0)$ in the form
\begin{eqnarray*}
\ ^{\alpha }\Gamma _{\beta \gamma }^{\alpha }&=&\left( \ \ ^{\alpha
}L_{jk}^{i},\ \ ^{\alpha }L_{jk}^{a},\ \ ^{\alpha }L_{bk}^{i},\ \ ^{\alpha
}L_{bk}^{a},\ \ ^{\alpha }C_{jb}^{i},\ \ ^{\alpha }C_{jb}^{a},\ \ ^{\alpha
}C_{bc}^{i},\ \ ^{\alpha }C_{bc}^{a}\right) , \\
\mbox{ where } 
 && \ ^{\alpha }
 \bigtriangledown _{\ ^{\alpha }\mathbf{e}_{k}}(\ ^{\alpha }%
\mathbf{e}_{j}) = \ ^{\alpha }L_{jk}^{i}\ ^{\alpha }\mathbf{e}_{i}+\ ^{\alpha }L_{jk}^{a}\ ^{\alpha }e_{a}, \\
 &&\ ^{\alpha }\bigtriangledown _{\ ^{\alpha }\mathbf{e}_{k}}(\ ^{\alpha
}e_{b}) =\ ^{\alpha }L_{bk}^{i}\ ^{\alpha }\mathbf{e}_{i}+\ ^{\alpha
}L_{bk}^{a}\ ^{\alpha }e_{a}, \\
&&\ ^{\alpha }\bigtriangledown _{\ ^{\alpha }e_{b}}(\ ^{\alpha }\mathbf{e}%
_{j}) =\ \ ^{\alpha }C_{jb}^{i}\ ^{\alpha }\mathbf{e}_{i}+\ \ ^{\alpha
}C_{jb}^{a}\ ^{\alpha }e_{a}, \\
 &&\ ^{\alpha }\bigtriangledown _{\ ^{\alpha }e_{c}}(\ ^{\alpha }e_{b}) =\
\ ^{\alpha }C_{bc}^{i}\ ^{\alpha }\mathbf{e}_{i}+\ ^{\alpha }C_{bc}^{a}\
^{\alpha }e_{a}.
\end{eqnarray*}%
Following a straightforward fractional coefficient computation, we can prove

\begin{corollary}
With respect to N--adapted fractional bases (\ref{dder}) and (\ref{ddif}),
the coefficients of the fractional Levi--Civita and canonical d--connection
satisfy the distorting relations
\begin{equation}
\ \ \ ^{\alpha }\Gamma _{\ \alpha \beta }^{\gamma }=\ \ ^{\alpha }\widehat{%
\mathbf{\Gamma }}_{\ \alpha \beta }^{\gamma }+\ \ \ ^{\alpha }Z_{\ \alpha
\beta }^{\gamma }  \label{cdeft}
\end{equation}%
where the explicit components of distortion tensor $\ _{\shortmid }Z_{\
\alpha \beta }^{\gamma }$ \ are computed%
\begin{eqnarray}
\ \ \ ^{\alpha }Z_{jk}^{i} &=&0,\ \ \ ^{\alpha }Z_{jk}^{a}=-\ \ ^{\alpha
}C_{jb}^{i}\ \ ^{\alpha }g_{ik}\ \ ^{\alpha }g^{ab}-\frac{1}{2}\ \ ^{\alpha
}\Omega _{jk}^{a},  \notag \\
\ \ \ ^{\alpha }Z_{bk}^{i} &=&\frac{1}{2}\ \ ^{\alpha }\Omega _{jk}^{c}\ \
^{\alpha }g_{cb}\ \ ^{\alpha }g^{ji}-\frac{1}{2}(\delta _{j}^{i}\delta
_{k}^{h}-\ \ ^{\alpha }g_{jk}\ \ ^{\alpha }g^{ih})\ \ ^{\alpha }C_{hb}^{j},
\notag \\
\ \ \ ^{\alpha }Z_{bk}^{a} &=&\frac{1}{2}(\delta _{c}^{a}\delta _{d}^{b}+\ \
^{\alpha }g_{cd}\ \ ^{\alpha }g^{ab})\left[ \ \ ^{\alpha }L_{bk}^{c}-\ \
^{\alpha }e_{b}(\ \ ^{\alpha }N_{k}^{c})\right] ,  \notag \\
\ \ \ ^{\alpha }Z_{kb}^{i} &=&\frac{1}{2}\ \ ^{\alpha }\Omega _{jk}^{a}\ \
^{\alpha }g_{cb}\ \ ^{\alpha }g^{ji}+\frac{1}{2}(\delta _{j}^{i}\delta
_{k}^{h}-\ \ ^{\alpha }g_{jk}\ \ ^{\alpha }g^{ih})\ \ ^{\alpha }C_{hb}^{j},
\notag \\
\ \ \ ^{\alpha }Z_{jb}^{a} &=&-\frac{1}{2}(\delta _{c}^{a}\delta _{b}^{d}-\
\ ^{\alpha }g_{cb}\ \ ^{\alpha }g^{ad})\left[ \ \ ^{\alpha }L_{dj}^{c}-\ \
^{\alpha }e_{d}(\ \ ^{\alpha }N_{j}^{c})\right] ,\   \label{cdeftc} \\
\ \ ^{\alpha }Z_{bc}^{a} &=&0,  \notag \\
\ \ \ ^{\alpha }Z_{ab}^{i} &=&-\frac{\ \ ^{\alpha }g^{ij}}{2}\{\left[ \ \
^{\alpha }L_{aj}^{c}-\ \ ^{\alpha }e_{a}(\ \ ^{\alpha }N_{j}^{c})\right] \ \
^{\alpha }g_{cb}  \notag \\
&&+\left[ \ \ ^{\alpha }L_{bj}^{c}-\ \ ^{\alpha }e_{b}(\ \ ^{\alpha
}N_{j}^{c})\right] \ \ ^{\alpha }g_{ca}\}.  \notag
\end{eqnarray}
\end{corollary}

We emphasize that there are not simple relations of type  (\ref{cdeft}) and (%
\ref{cdeftc}) if the fractional integro--differential structure would be not
elaborated in N--adapted form for the left Caputo derivative. For the
fractional RL derivatives, it is not possible to introduce N--anholonomic
distributions when the formulas would preserve a maximal similarity with the
integer nonholonomic case.

\section{Perelman Type Fractional Functionals}

\label{ssfgpf} The goal of this section is to show that there is a
fractional integro--differential calculus admitting generalizations of the
Hamilton--Perelman Ricci flow evolution theory. Proofs are simplified for
correspondingly defined nonholonomic fractional distributions.

\subsection{On (non) holonomic Ricci flows}

For Riemannian spaces of integer dimension, the Grisha Perelman's
fundamental idea was to prove that the Ricci flow is not only a gradient
flow but, introducing two Lyapunov type functionals, can be defined also\ as
a dynamical system on the spaces of Riemannian metrics.

The Ricci flow equation was postulated by R. Hamilton \cite{ham1} as an
evolution equation\footnote{%
for our further purposes, it is convenient to use a different system of
denotations than those considered by R. Hamilton or Grisha Perelman on
holonomic Riemannian spaces of integer dimensions}
\begin{equation}
\frac{\partial g_{\alpha \beta }(\chi )}{\partial \chi }=-2\ R_{\alpha \beta
}(\chi )  \label{heq1}
\end{equation}%
for a set of Riemannian metrics $g_{\alpha \beta }(\chi )$ and corresponding
Ricci tensors $\ R_{\alpha \beta }(\chi ),$ derived for corresponding set of
Levi--Civita connections $\nabla (\chi ),$ all parametrized by a real
parameter $\chi .$

The Perelman's functionals were introduced for Ricci flows of Riemannian
metrics. For the Levi--Civita- connection $\nabla $ defined by a metric $%
\mathbf{g,}$ such fundamental functionals are written in the form
\begin{eqnarray}
\ _{\shortmid }\mathcal{F}(f) &=&\int\limits_{\mathbf{V}}\left( \
_{\shortmid }R+\left| \nabla f\right| ^{2}\right) e^{-f}\ dV,  \label{pfrs}
\\
\ _{\shortmid }\mathcal{W}(f,\tau ) &=&\int\limits_{\mathbf{V}}\left[ \tau
\left( \ _{\shortmid }R+\left| \nabla f\right| \right) ^{2}+f-\frac{n+m}{2}%
\right] \mu \ dV,  \notag
\end{eqnarray}%
where $dV$ is the volume form of $\ \mathbf{g,}$ integration is taken over
compact $\mathbf{V}$ and $\ _{\shortmid }R$ is the scalar curvature computed
for $\nabla .$ For a flow parameter $\tau >0,$ we have $\int\nolimits_{%
\mathbf{V}}\mu dV=1$ when $\ \mu =\left( 4\pi \tau \right)
^{-(n+m)/2}e^{-f}. $

In our works \cite{vnhrf1,vnhrf2,vnhrf3,nhrfs1,nhrfs2,nhrfs3}, we proved
that nonholonomic Ricci flows of the Lagrange--Finsler geometries and
various generalizations with nonsymmetric metics, noncommutative structures
etc, can be modelled as constrained structures on N--anholonomic Riemannian
spaces. The main conclusion was that following a N--adapted formalism, the
Ricci flow theory can be extended for non--Riemannian geometries.

\subsection{Fractional functionals for nonholonomic Ricci flows}

The functional approach can be redefined for N--anholonomic manifolds, for
our purposes, modeled as fractional spaces $\overset{\alpha }{\mathbf{V}}.$
Fractional flows are considered with fractional derivative on parameters. In
N--adapted form, we follow the methods from \cite{vnhrf1,vnhrf2} extended
for fractional derivatives.

\begin{claim}
For fractional nonholonomic geometries defined by the canonical
d--connection $\ ^{\alpha }\widehat{\mathbf{D}},$ the fractional versions of
Perelman's functionals are%
\begin{eqnarray}
\ ^{\alpha }\widehat{\mathcal{F}}(\ ^{\alpha }\mathbf{g},\overset{\alpha }{%
\mathbf{N}},\ ^{\alpha }\widehat{f}) &=&\int\limits_{\ ^{\alpha }\mathbf{V}%
}(\ ^{\alpha }R+\ ^{\alpha }S+|\ ^{\alpha }\widehat{\mathbf{D}}\ ^{\alpha }%
\widehat{f}|^{2})e^{-\ ^{\alpha }\widehat{f}}\ ^{\alpha }d\ ^{\alpha }V,
\label{npf1} \\
\ ^{\alpha }\widehat{\mathcal{W}}(\ ^{\alpha }\mathbf{g,}\ \overset{\alpha }{%
\mathbf{N}},\ ^{\alpha }\widehat{f},\ ^{\alpha }\tau ) &=&\int\limits_{\
^{\alpha }\mathbf{V}}[\ ^{\alpha }\widehat{\tau }(\ ^{\alpha }R+\ ^{\alpha
}S+|\ _{h}^{\alpha }D\ ^{\alpha }\widehat{f}|+|\ _{v}^{\alpha }D\ ^{\alpha }%
\widehat{f}|)^{2}  \notag \\
&&+\ ^{\alpha }\widehat{f}-\frac{n+m}{2}]\ ^{\alpha }\widehat{\mu }\
^{\alpha }d\ ^{\alpha }V,  \label{npf2}
\end{eqnarray}%
where $\ ^{\alpha }d\ ^{\alpha }V$ is the volume fractional form of $\ \
^{\alpha }\mathbf{g}$ (\ref{m1}), $\ ^{\alpha }R$ and $\ ^{\alpha }S$ \ are
respectively the h- and v--components of the curvature scalar (\ref{sdccurv}%
) of $\ ^{\alpha }\widehat{\mathbf{D}},$ for $\ ^{\alpha }\widehat{\mathbf{D}%
}_{\beta }=(\ ^{\alpha }D_{i},\ ^{\alpha }D_{a}),$ or $\ \ ^{\alpha }%
\widehat{\mathbf{D}}=(\ ^{h}D,\ ^{v}D),$ $\left| \ ^{\alpha }\widehat{%
\mathbf{D}}\ ^{\alpha }\widehat{f}\right| ^{2}=\left| \ _{h}^{\alpha }D\
^{\alpha }\widehat{f}\right| ^{2}+\left| \ _{v}^{\alpha }D\ ^{\alpha }%
\widehat{f}\right| ^{2},$ and $\ ^{\alpha }\widehat{f}$ satisfies $%
\int\nolimits_{\ ^{\alpha }\mathbf{V}}\ ^{\alpha }\widehat{\mu }\ ^{\alpha
}d\ ^{\alpha }V=1$ \ for $\ ^{\alpha }\widehat{\mu }=\left( 4\pi \ \tau
\right) ^{-(n+m)/2}e^{-\ ^{\alpha }\widehat{f}}$ and fractional flow
parameter $\ \tau >0.$
\end{claim}

\begin{proof}
Formulas (\ref{pfrs}) can be rewritten for some fractional functions $\
^{\alpha }\widehat{f}$ and $\ ^{\alpha }f$ when
\begin{eqnarray*}
(\ _{\shortmid }^{\alpha }R+\left| \ ^{\alpha }\nabla \ ^{\alpha }f\right|
^{2}) e^{-\ ^{\alpha }f} &=&(\ ^{\alpha }R+\ ^{\alpha }S+\left| \
_{h}^{\alpha }D\ ^{\alpha }\widehat{f}\right| ^{2}+\left| \ _{v}^{\alpha }D\
^{\alpha }\widehat{f}\right| ^{2}) e^{-\ ^{\alpha }\widehat{f}}\  \\
&&+\ ^{\alpha }\Phi
\end{eqnarray*}%
for a re--scaling of fractional parameter $\ \tau \rightarrow \ \widehat{%
\tau }$ with
\begin{eqnarray*}
&&\left[ \ ^{\alpha }\tau \left( \ _{\shortmid }^{\alpha }R+\left| \
^{\alpha }\nabla \ ^{\alpha }f\right| \right) ^{2}+\ ^{\alpha }f-\frac{n+m}{2%
})\right] \ ^{\alpha }\mu = \\
&&\left[ \ ^{\alpha }\widehat{\tau }\left( \ ^{\alpha }R+\ ^{\alpha
}S+\left| \ _{h}^{\alpha }D\widehat{f}\right| +\left| \ _{v}^{\alpha }D%
\widehat{f}\right| \right) ^{2}+\ ^{\alpha }\widehat{f}-\frac{n+m}{2}\right]
\ ^{\alpha }\widehat{\mu }+\ ^{\alpha }\Phi _{1},
\end{eqnarray*}%
for some $\ ^{\alpha }\Phi $ and $\ ^{\alpha }\Phi _{1}$ for which $%
\int\limits_{\ ^{\alpha }\mathbf{V}}\ ^{\alpha }\Phi \ ^{\alpha }d\ ^{\alpha
}V=0$ and $\int\limits_{\ ^{\alpha }\mathbf{V}}\ ^{\alpha }\Phi _{1}\
^{\alpha }d\ ^{\alpha }V=0.$ $\square $
\end{proof}

For proofs of the Main Results in section \ref{sfhe}, the next lemma will be
important.

\begin{lemma}
\label{lem1}The first N--adapted fractional variations of (\ref{npf1}) are
given by
\begin{eqnarray}
&&\delta \ ^{\alpha }\widehat{\mathcal{F}}(v_{ij},v_{ab},\ ^{h}f,\ ^{v}f)=
\label{vnpf1} \\
&&\int\limits_{\mathbf{V}}\{[-v_{ij}(\ ^{\alpha }R_{ij}+\ ^{\alpha }D_{i}\
^{\alpha }D_{j}\ ^{\alpha }\widehat{f})+(\frac{\ ^{h}v}{2}-\ ^{h}f)(2\
^{h}\Delta \widehat{f}-|\ ^{h}D\ \widehat{f}|)+\ ^{\alpha }R]  \notag \\
&&+[-v_{ab}(\ ^{\alpha }R_{ab}+\ ^{\alpha }D_{a}\ ^{\alpha }D_{b}\ ^{\alpha }%
\widehat{f})+(\frac{\ ^{v}v}{2}-\ ^{v}f)\left( 2\ ^{v}\Delta \widehat{f}-|\
^{v}D\ \widehat{f}|\right)  \notag \\
&&+\ ^{\alpha }S]\}e^{-\ ^{\alpha }\widehat{f}}\ ^{\alpha }d\ ^{\alpha }V,
\notag
\end{eqnarray}%
where $^{h}\Delta =\ ^{\alpha }D_{i}\ ^{\alpha }D^{i}$ and $\ ^{v}\Delta =\
^{\alpha }D_{a}\ ^{\alpha }D^{a},\widehat{\Delta }=$ $\ ^{h}\Delta +\
^{v}\Delta ,$ and $\ ^{h}v=\ ^{\alpha }g^{ij}v_{ij},\ ^{v}v=\ ^{\alpha
}g^{ab}v_{ab};$ \ for h--variation $\ ^{h}\delta \ ^{\alpha }g_{ij}=v_{ij},$
v-variation $\ ^{v}\delta \ ^{\alpha }g_{ab}=v_{ab}$ and variations $\
^{h}\delta \ ^{\alpha }\widehat{f}=\ ^{h}f,$ $^{v}\delta \ ^{\alpha }%
\widehat{f}=\ ^{v}f.$
\end{lemma}

\begin{proof}
We fix a N--connection structure $\ \overset{\alpha }{\mathbf{N}}$ for a
fractional metric $\ ^{\alpha }\mathbf{g}$ (\ref{m1}). Then we follow a
N--adapted fractional calculus similar to that for Perelman's Lemma in \cite%
{gper1}. We omit details given, for instance, in the proof for integer
configurations in \cite{caozhu}, see there Lemma 1.5.2. $\square $
\end{proof}

\section{Fractional Hamilton's Evolution Equations}

\label{sfhe}

In this section, we formulate the main results of this paper, on fractional
Ricci flow theory: we sketch the proofs that evolution of fractional
geometries can be derived by variation of generalized Perelman functionals
and show that a statistical analogy can be provided to such fractional flow
processes. For integer dimensions, such constructions model holonomic Ricci
flows of (pseudo) Riemannian and K\"{a}hler geometries \cite%
{gper1,caozhu,kleiner,rbook}.

A heuristic approach to develop a fractional Ricci flow theory is to take the equations
\begin{equation}
\frac{\partial }{\partial \chi }g_{\underline{\alpha }\underline{\beta }%
}=-2\ _{\shortmid }R_{\underline{\alpha }\underline{\beta }}+\frac{2r}{5}g_{%
\underline{\alpha }\underline{\beta }},  \label{feq}
\end{equation}%
describing normalized (holonomic) Ricci flows with respect to a coordinate
base $\partial _{\underline{\alpha }}=\partial /\partial u^{\underline{%
\alpha }}.$\footnote{%
In this integer case, we underline the indices with respect to the
coordinate bases in order to distinguish them from those defined with
respect to the 'N--elongated' local bases (\ref{dder}) and (\ref{ddif}).} In
(\ref{feq}), the normalizing factor $r=\int \ _{\shortmid }RdV/dV$ is
introduced in order to preserve the volume $V;$ $\ _{\shortmid }R_{%
\underline{\alpha }\underline{\beta }}$ and $\ _{\shortmid }R=g^{\underline{%
\alpha }\underline{\beta }}\ _{\shortmid }R_{\underline{\alpha }\underline{%
\beta }}$ are computed for the Levi--Civita connection $\nabla .$ Then we
change the geometric objects (tensors, derivatives and parameter) into
fractional ones, and obtain a non--integer generalization of Hamilton's
equations,
\begin{eqnarray}
\ _{\ _{1}\chi }\overset{\alpha }{\underline{\partial }}_{\chi }\ ^{\alpha
}g_{ij} &=&2[\ ^{\alpha }N_{i}^{a}\ ^{\alpha }N_{j}^{b}\ (\ _{\shortmid
}^{\alpha }R_{ab}-\ ^{\alpha }\lambda \ ^{\alpha }g_{ab})-\ _{\shortmid
}^{\alpha }R_{ij}+\ ^{\alpha }\lambda \ ^{\alpha }g_{ij}]  \notag \\
&&-\ ^{\alpha }g_{cd}\ _{\ _{1}\chi }\overset{\alpha }{\underline{\partial }}%
_{\chi }(\ ^{\alpha }N_{i}^{c}\ ^{\alpha }N_{j}^{d}),  \label{eq1} \\
\ _{\ _{1}\chi }\overset{\alpha }{\underline{\partial }}_{\chi }\ ^{\alpha
}g_{ab} &=&-2\ _{\shortmid }^{\alpha }R_{ab}+2\ ^{\alpha }\lambda \ ^{\alpha
}g_{ab},\   \label{eq2} \\
\ _{\ _{1}\chi }\overset{\alpha }{\underline{\partial }}_{\chi }(\ ^{\alpha
}N_{j}^{e}\ ^{\alpha }g_{ae})&=&-2\ _{\shortmid }^{\alpha }R_{ia}+2\
^{\alpha }\lambda \ ^{\alpha }N_{j}^{e}\ \ ^{\alpha }g_{ae},  \label{eq3}
\end{eqnarray}%
where $\ ^{\alpha }\lambda =\ ^{\alpha }r/5,$ with $\ ^{\alpha }r=\int \
_{\shortmid }^{\alpha }R\ ^{\alpha }d\ ^{\alpha }V/\ ^{\alpha }d\ ^{\alpha
}V,$ and the metric coefficients are those for (\ref{fmcf}) parametrized by
ansatz (\ref{qel}), with respect to a fractional local coordinate basis (\ref%
{frlccb}).

A fractional differential geometry is modelled by nonholonomic
integro--differential structures. A self--consistent system of fractional
equations has to be N--adapted. We change in (\ref{eq1})--(\ref{eq3}) the
corresponding values: $\ ^{\alpha }\nabla \rightarrow \ ^{\alpha }\widehat{%
\mathbf{D}}$ and $\ _{\shortmid }^{\alpha }R_{\alpha \beta }\rightarrow \
^{\alpha }\widehat{\mathbf{R}}_{\alpha \beta }.$ The resulting N--adapted
fractional evolution equations for Ricci flows of symmetric fractional
metrics, with respect to local coordinate frames (\ref{frlccb}), are
\begin{eqnarray}
\ _{\ _{1}\chi }\overset{\alpha }{\underline{\partial }}_{\chi }\ ^{\alpha
}g_{ij} &=&2[\ ^{\alpha }N_{i}^{a}\ ^{\alpha }N_{j}^{b}\ (\ ^{\alpha }%
\underline{\widehat{R}}_{ab}-\ ^{\alpha }\lambda \ ^{\alpha }g_{ab})-\
^{\alpha }\underline{\widehat{R}}_{ij}+ \ ^{\alpha }\lambda \
^{\alpha}g_{ij}]  \notag \\
&&-\ ^{\alpha }g_{cd}\ _{\ _{1}\chi }\overset{\alpha }{\underline{\partial }}%
_{\chi }(\ ^{\alpha }N_{i}^{c}\ ^{\alpha }N_{j}^{d}),  \label{e1} \\
\ _{\ _{1}\chi }\overset{\alpha }{\underline{\partial }}_{\chi }\ ^{\alpha
}g_{ab} &=&-2\ \left( \ ^{\alpha }\underline{\widehat{R}}_{ab}-\ ^{\alpha
}\lambda \ ^{\alpha }g_{ab}\right) ,\   \label{e2} \\
\ \ ^{\alpha }\widehat{R}_{ia} &=&0\mbox{ and }\ \ ^{\alpha }\widehat{R}%
_{ai}=0,  \label{e3}
\end{eqnarray}%
where the fractional Ricci tensor coefficients $\ ^{\alpha }\underline{%
\widehat{R}}_{ij}$ and $\ ^{\alpha }\underline{\widehat{R}}_{ab}$ are
computed with respect to coordinate coframes (\ref{frlccb}), being frame
transforms (\ref{fr}) of the corresponding formulas (\ref{dricci}) defined
with respect to N--adapted coframes (\ref{ddif}). The equations (\ref{e3})
constrain the nonholonomic fractional Ricci flows to result in symmetric
fractional metrics. In general, fractional geometries are with nonholonomic
integro--differential structures resulting in nonsymmetric fractional
metrics, a similar conclusion for integer dimensions was proven in Ref. \cite%
{vnhrf3}.

\subsection{Main Theorems on fractional Ricci flows}

One of the most important Perelman's contributions to the theory of Ricci
flows was that he proved that Hamilton's evolution equations, in some
adapted forms, can be derived from certain functionals following a
variational procedure. We show that Perelman's approach can be generalized
to a fractional N--adapted formalism for evolution of geometric objects. In
explicit form, we show how equations of type (\ref{e1}) and (\ref{e2}) can
be derived by a fractional integro--differential calculus (for simplicity,
we take a zero normalized term with $\ ^{\alpha }\lambda =0).$

\begin{definition}
A general fractional metric $\ ^{\alpha }\mathbf{g}$ evolving via a general
fractional Ricci flow is called a breather if for some $\chi _{1}<\chi _{2}$
and $\beta >0$ the metrics $\beta \ ^{\alpha }\mathbf{g(}\chi _{1}\mathbf{)}$
and $\beta \ ^{\alpha }\mathbf{g(}\chi _{2}\mathbf{)}$ differ only by a
fractional diffeomorphism preserving the Whitney sum (\ref{whit}). The
cases $\beta =,<,>1$ define correspondingly the steady, shrinking and
expanding breathers.
\end{definition}

We note that because of nonholonomic character of fractional evolution we
can model processes when, for instance, the h--component of metric is steady
but the v--component is shrinking. Clearly, the expending properties depend
on the type of \ calculus and connections are used for definition of Ricci
flows.

Following a N--adapted variational calculus for $\ ^{\alpha }\widehat{%
\mathcal{F}}(\ ^{\alpha }\mathbf{g},\overset{\alpha }{\mathbf{N}},\ ^{\alpha
}\widehat{f}),$ see Lemma \ref{lem1}, with Laplacian $\ ^{\alpha }\widehat{%
\Delta }$ and h- and v--components of the Ricci tensor, $\ ^{\alpha }%
\widehat{R}_{ij}$ and $\ ^{\alpha }\widehat{S}_{ij},$ defined by $\ \
^{\alpha }\widehat{\mathbf{D}}$ and considering parameter $\tau (\chi ),$ $%
\partial \tau /\partial \chi =-1,$ we prove

\begin{theorem}
The fractional N--adapted Ricci flows are characterized by evolution
equations
\begin{eqnarray*}
\ _{\ _{1}\chi }\overset{\alpha }{\underline{\partial }}_{\chi }\ ^{\alpha }%
\underline{g}_{ij} &=&-2\ \ ^{\alpha }\underline{\widehat{R}}_{ij}, \ _{\
_{1}\chi }\overset{\alpha }{\underline{\partial }}_{\chi }\ \underline{g}%
_{ab}{\partial \chi }=-2\ \ ^{\alpha }\underline{\widehat{R}}_{ab}, \\
\ _{\ _{1}\chi }\overset{\alpha }{\underline{\partial }}_{\chi }\ ^{\alpha }%
\widehat{f} &=&-\ \ ^{\alpha }\widehat{\Delta }\ \ ^{\alpha }\widehat{f}%
+\left| \ \ ^{\alpha }\widehat{\mathbf{D}}\ \ ^{\alpha }\widehat{f}\right|
^{2}-\ \ ^{\alpha }R-\ \ ^{\alpha }S
\end{eqnarray*}%
and the properties that $\int\limits_{\ \ ^{\alpha }\mathbf{V}}e^{-\ \
^{\alpha }\widehat{f}}\ \ ^{\alpha }d\ \ ^{\alpha }V=const$ and
\begin{eqnarray*}
&&\ _{\ _{1}\chi }\overset{\alpha }{\underline{\partial }}_{\chi }\ ^{\alpha
}\widehat{\mathcal{F}}(\ ^{\alpha }\mathbf{g(}\chi \mathbf{),}\overset{%
\alpha }{\mathbf{N}}(\chi ),\ ^{\alpha }\widehat{f}(\chi )) =2\int\limits_{\
^{\alpha }\mathbf{V}}[|\ ^{\alpha }\widehat{R}_{ij}+\ ^{\alpha }D_{i}\
^{\alpha }D_{j}\ ^{\alpha }\widehat{f}|^{2} \\
&&+|\ ^{\alpha }\widehat{R}_{ab}+\ ^{\alpha }D_{a}\ ^{\alpha }D_{b}\
^{\alpha }\widehat{f}|^{2}]e^{-\ ^{\alpha }\widehat{f}}\ ^{\alpha }d\
^{\alpha }V.
\end{eqnarray*}
\end{theorem}

\begin{proof}
Such a proof which is very similar to those for Riemannian spaces,
originally proposed by G. Perelman \cite{gper1}, see also details in the
Proposition 1.5.3 of \cite{caozhu}, and nonholonomic manifolds (additional
remarks on the canonical d--connection $\widehat{\mathbf{D}}$ on
nonholonomic manifolds of integer dimension given in \cite{vnhrf1}). All
those constructions can be reproduced in N--adapted fractional form using $\
^{\alpha }\widehat{\mathbf{D}}$ $\ $with coefficients (\ref{candcon}).$%
\square $
\end{proof}

A similar analogy of calculus with $\ ^{\alpha }\widehat{\mathbf{D}}$ \ to
that for the integer case with $\nabla $ allows us to generalize the
formulation and proof of Proposition 1.5.8 in \cite{caozhu} containing the
details of the original result from \cite{gper1}), resulting in:

\begin{theorem}
\label{theveq}If a family of fractional metric $\ \ ^{\alpha }\mathbf{g}%
(\chi ),$ fractional function $\ \ ^{\alpha }\widehat{f}(\chi )$ and $\ $%
parameter function $\widehat{\tau }(\chi )$ evolve subjected to the
conditions of the system of equations%
\begin{eqnarray*}
\ _{\ _{1}\chi }\overset{\alpha }{\underline{\partial }}_{\chi }\ \ ^{\alpha
}\underline{g}_{ij} &=&-2\ \ ^{\alpha }\underline{\widehat{R}}_{ij},\ _{\
_{1}\chi }\overset{\alpha }{\underline{\partial }}_{\chi }\ \ ^{\alpha }%
\underline{g}_{ab}=-2\ \ ^{\alpha }\underline{\widehat{R}}_{ab}, \\
\ _{\ _{1}\chi }\overset{\alpha }{\underline{\partial }}_{\chi }\ \ \
^{\alpha }\widehat{f} &=&-\ \ ^{\alpha }\widehat{\Delta }\ \ ^{\alpha }%
\widehat{f}+\left| \ \ ^{\alpha }\widehat{\mathbf{D}}\ \ ^{\alpha }\widehat{f%
}\right| ^{2}-\ \ ^{\alpha }R-\ \ ^{\alpha }S+\frac{n+m}{2\widehat{\tau }},
\\
\ _{\ _{1}\chi }\overset{\alpha }{\underline{\partial }}_{\chi }\widehat{%
\tau } &=&-1,
\end{eqnarray*}%
there are satisfied the properties $\int\limits_{\ ^{\alpha }\mathbf{V}%
}(4\pi \widehat{\tau })^{-(n+m)/2}e^{-\ ^{\alpha }\widehat{f}}\ ^{\alpha }d\
^{\alpha }V=const$ and
\begin{eqnarray*}
&&\ _{\ _{1}\chi }\overset{\alpha }{\underline{\partial }}_{\chi }\ \
^{\alpha }\widehat{\mathcal{W}}(\ \ ^{\alpha }\mathbf{g}(\chi )\mathbf{,}%
\overset{\alpha }{\mathbf{N}}(\chi ),\ \ ^{\alpha }\widehat{f}(\chi ),%
\widehat{\tau }(\chi )) = \\
&& 2\int\limits_{\mathbf{V}}\widehat{\tau }[|\ \ ^{\alpha }\widehat{R}%
_{ij}+\ \ ^{\alpha }D_{i}\ \ ^{\alpha }D_{j}\ \ ^{\alpha }\widehat{f}-\frac{1%
}{2\widehat{\tau }}\ \ ^{\alpha }g_{ij}|^{2}+ \\
&&|\ \ ^{\alpha }\widehat{R}_{ab}+\ \ ^{\alpha }D_{a}\ \ ^{\alpha }D_{b}\ \
^{\alpha }\widehat{f}-\frac{1}{2\widehat{\tau }}\ \ ^{\alpha
}g_{ab}|^{2}](4\pi \widehat{\tau })^{-(n+m)/2}e^{-\ \ ^{\alpha }\widehat{f}%
}\ \ ^{\alpha }d\ \ ^{\alpha }V.
\end{eqnarray*}
\end{theorem}

The functional $\ \ ^{\alpha }\mathcal{W}(\ \ ^{\alpha }\mathbf{g}(\chi )%
\mathbf{,}\overset{\alpha }{\mathbf{N}}(\chi ),\ ^{\alpha }f(\chi ),\tau
(\chi ))$ is nondecreasing in time and the monotonicity is strict unless we
are on a shrinking fractional gradient soliton. This property depends on the
type of fractional d--connection, or covariant connection we use.

\begin{corollary}
The fractional evolution, for all time $\tau \in \lbrack 0,\tau _{0}),$ of \
N--adapted frames
\begin{equation*}
\ \ \ ^{\alpha }\mathbf{e}_{\alpha }(\tau )=\ \ ^{\alpha }\mathbf{e}_{\alpha
}^{\ \underline{\alpha }}(\tau ,u)\ \ ^{\alpha }\partial _{\underline{\alpha
}}
\end{equation*}%
is defined by the coefficients
\begin{equation*}
\ \ \ ^{\alpha }\mathbf{e}_{\alpha }^{\ \underline{\alpha }}(\tau ,u)=\left[
\begin{array}{cc}
\ \ \ ^{\alpha }e_{i}^{\ \underline{i}}(\tau ,u) & ~\ \ ^{\alpha
}N_{i}^{b}(\tau ,u)\ \ \ ^{\alpha }e_{b}^{\ \underline{a}}(\tau ,u) \\
0 & \ \ ^{\alpha }\ e_{a}^{\ \underline{a}}(\tau ,u)%
\end{array}%
\right] ,\
\end{equation*}%
with
\begin{equation*}
\ \ \ ^{\alpha }g_{ij}(\tau )=\ \ \ ^{\alpha }e_{i}^{\ \underline{i}}(\tau
,u)\ \ \ ^{\alpha }e_{j}^{\ \underline{j}}(\tau ,u)\eta _{\underline{i}%
\underline{j}},
\end{equation*}%
where $\eta _{\underline{i}\underline{j}}=diag[\pm 1,...\pm 1]$ establish
the signature of $\ ^{\alpha }g_{\underline{\alpha }\underline{\beta }%
}^{[0]}(u),$ is given by equations
\begin{equation*}
_{\ _{1}\tau }\overset{\alpha }{\underline{\partial }}_{\tau }\ ^{\alpha
}e_{\alpha }^{\ \underline{\alpha }}=\ ^{\alpha }g^{\underline{\alpha }%
\underline{\beta }}\ ~_{\shortmid }^{\alpha }R_{\underline{\beta }\underline{%
\gamma }}~\ \ \ ^{\alpha }e_{\alpha }^{\ \underline{\gamma }},
\end{equation*}%
if we prescribe fractional flows for the Levi--Civita connection $\ ^{\alpha
}\nabla ,$ and
\begin{equation*}
_{\ _{1}\tau }\overset{\alpha }{\underline{\partial }}_{\tau }\ ^{\alpha
}e_{\alpha }^{\ \underline{\alpha }}=\ ^{\alpha }g^{\underline{\alpha }%
\underline{\beta }}\ ^{\alpha }\widehat{R}_{\underline{\beta }\underline{%
\gamma }}\ ^{\alpha }e_{\alpha }^{\ \underline{\gamma }},
\end{equation*}%
if we prescribe fractional flows for the canonical d--connection $\ ^{\alpha
}\widehat{\mathbf{D}}$.
\end{corollary}

We conclude that the fractional flows are  characterized additionally by
fractional evolutions of N--adapted frames (\ref{fr}) (see a similar proof
for flows of integer dimension nonholonomic frames in \cite{vnhrf1}).

\subsection{Statistical Analogy for Fractional Ricci Flows}

A functional $\ ^{\alpha }\widehat{\mathcal{W}}$ \ is in a\ sense analogous
to minus entropy (such an interpretation  was supposed by Grisha Perelman for Ricci flows of
Riemannian metrics \cite{gper1}). This allows us to elaborate a statistical
model for fractional nonholonomic flows if the conditions of Theorem \ref%
{theveq} are satisfied.\footnote{%
Let us remember some concepts from statistical mechanics. The partition
function $Z=\int \exp (-\beta E)d\omega (E)$ for the canonical ensemble at
temperature $\beta ^{-1}$ is defined by the measure taken to be the density
of states $\omega (E).$ The thermodynamical values are computed in the form:
the average energy, $<E>=-\partial \log Z/\partial \beta ,$ the entropy $%
S=\beta <E>+\log Z$ and the fluctuation $\sigma =<\left( E-<E>\right)
^{2}>=\partial ^{2}\log Z/\partial \beta ^{2}.$}

For the partition function $\ ^{\alpha }\widehat{Z}=\exp \{\int\nolimits_{\ ^{\alpha }\mathbf{V}%
}[-\ ^{\alpha }\widehat{f}+\frac{n+m}{2}]\ ^{\alpha }\widehat{\mu }\
^{\alpha }d\ ^{\alpha }V\},$ we prove:

\begin{theorem}
Any family of fractional nonholonomic geometries satisfying the fractional
evolution equations for the canonical d--connection is characterized by \
three thermodynamic values
\begin{eqnarray*}
\left\langle \ ^{\alpha }\widehat{E}\right\rangle &=&-\widehat{\tau }%
^{2}\int\limits_{\ ^{\alpha }\mathbf{V}}(\ ^{\alpha }R+\ ^{\alpha }S+\left|
\ _{h}^{\alpha }D\ ^{\alpha }\widehat{f}\right| ^{2}+\left| \ _{v}^{\alpha
}D\ ^{\alpha }\widehat{f}\right| ^{2}-\frac{n+m}{2\widehat{\tau }}) \\
&&\times \ ^{\alpha }\widehat{\mu }\ ^{\alpha }d\ ^{\alpha }V, \\
\ ^{\alpha }\widehat{S} &=&-\int\limits_{\ ^{\alpha }\mathbf{V}}[ \widehat{%
\tau }\left( \ ^{\alpha }R+\ ^{\alpha }S+\left| ^{h}D\ ^{\alpha }\widehat{f}%
\right| ^{2}+\left| ^{v}D\ ^{\alpha }\widehat{f}\right| ^{2}\right) +\
^{\alpha }\widehat{f} \\
&&-\frac{n+m}{2}]\ ^{\alpha }\widehat{\mu }\ \ ^{\alpha }d\ ^{\alpha }V, \\
\ ^{\alpha }\widehat{\sigma } &=&2\ \widehat{\tau }^{4}\int\limits_{\mathbf{V%
}}[|\ ^{\alpha }\widehat{R}_{ij}+\ ^{\alpha }D_{i}\ ^{\alpha }D_{j}\
^{\alpha }\widehat{f}-\frac{1}{2\widehat{\tau }}\ ^{\alpha }g_{ij}|^{2} \\
&&+|\ ^{\alpha }\widehat{R}_{ab}+\ ^{\alpha }D_{a}\ ^{\alpha }D_{b}\
^{\alpha }\widehat{f}-\frac{1}{2\widehat{\tau }}\ ^{\alpha }g_{ab}|^{2}]\
^{\alpha }\widehat{\mu }\ \ ^{\alpha }d\ ^{\alpha }V.
\end{eqnarray*}
\end{theorem}

A fractional, or integer, differential geometry defined by corresponding
fundamental geometric objects and a fixed, in general, non--integer
differential system is thermodynamically more convenient in dependence of
the values of the above mentioned characteristics of Ricci flow evolution.

\begin{conclusion}
Finally, we draw the conclusions:

\begin{itemize}
\item There is a version of fractional differential and integral calculus
based on the left Caputo derivative when the resulting models of fractional
differential geometry are with a N--connection adapted calculus similarly to noholonomic manifolds and  Finsler--Lagrange geometry.

\item A Ricci flow theory of fractional geometries can be considered as a
nonholonomic evolution model transforming standard integer metrics and
connections (for instance, in Riemann geometry) into generalized ones on
nonsymmetric/noncommutative/fractional ... spaces.

\item A very important property of fractional calculus theories and related
geometric and physical models is that we can work with more "singular"
 functions and field interaction/evolution models in physics and
applied mathematics.
\end{itemize}
\end{conclusion}

\setcounter{equation}{0} \renewcommand{\theequation}
{A.\arabic{equation}} \setcounter{subsection}{0}
\renewcommand{\thesubsection}
{A.\arabic{subsection}}

\appendix

\section{Fractional Integro--Differential Calculus on $\mathbb{R}^{n}$}

\label{asubs1}

We summarize the formalism for a vector fractional differential and integral
calculus elaborated on ''flat'' spaces \cite{taras1}. The constructions
involve a fundamental theorem of calculus and fractional integral Green's,
Stokes' and Gauss's theorems, which are important for definition, in this
work, of fractional Perelman's functionals.

\subsection{Riemann--Liouville and Caputo fractional derivatives}

It is possible to elaborate different types of models of fractional geometry
using different types of fractional derivatives. We follow an approach when
the geometric constructions are most closed to ''integer'' calculus.

\subsubsection{Left and right fractional RL derivatives}

Let us consider that $f(x)$ is a derivable function $f:[\ _{1}x,\
_{2}x]\rightarrow \mathbb{R},$ for $\mathbb{R\ni }\alpha >0,$ and denote the
derivative on $x$ as $\partial _{x}=\partial /\partial x.$

The left Riemann--Liouville (RL) derivative is
\begin{equation*}
\ \ _{\ _{1}x}\overset{\alpha }{\partial }_{x}f(x):=\frac{1}{\Gamma
(s-\alpha )}\left( \frac{\partial }{\partial x}\right) ^{s}\int\limits_{\ \
_{1}x}^{x}(x-\ x^{\prime })^{s-\alpha -1}f(x^{\prime })dx^{\prime },
\end{equation*}%
where $\Gamma $ is the Euler's gamma function. The left fractional Liouville
derivative of order $\alpha ,$ where $\ s-1<\alpha <s,$ with respect to
coordinate $x$ is defined
 $\ \overset{\alpha }{\partial }_{x}f(x):=\lim_{_{1}x\rightarrow -\infty }\ \
_{\ _{1}x}\overset{\alpha }{\partial }_{x}f(x).$ 

The right RL derivative is
 $\ _{\ x}\overset{\alpha }{\partial }_{\ _{2}x}f(x):=\frac{1}{\Gamma
(s-\alpha )}\left( -\frac{\partial }{\partial x}\right)
^{s}\int\limits_{x}^{\ _{2}x}(x^{\prime }-x)^{s-\alpha -1}$ $f(x^{\prime
})dx^{\prime }.$ %
The corresponding right fractional Liouville derivative is
 $\ _{\ x}\overset{\alpha }{\partial }f(x^{k}):=\lim_{_{2}x\rightarrow \infty
}\ \ _{\ x}\overset{\alpha }{\partial }_{\ _{2}x}f(x).$ In this work, we shall not use right derivatives.

Only the fractional Liouville derivatives define operators satisfying the
semigroup properties on function spaces. The fractional RL derivative of a
constant $C$ is not zero but, for instance, $\ _{\ _{1}x}\overset{\alpha }{%
\partial }_{x}C=C\frac{(x-\ _{1}x)^{-\alpha }}{\Gamma (1-\alpha )}.$
Complete fractional integro--differential constructions based only on such
derivatives seem to be very cumbersome and has a number of properties which
are very different from similar ones for integer calculus.

\subsubsection{Fractional Caputo derivatives}

The respective left and right fractional Caputo derivatives are 
\begin{eqnarray}
 \ _{\ _{1}x}\overset{\alpha }{\underline{\partial }}_{x}f(x)&:=&\frac{1}{%
\Gamma (s-\alpha )}\int\limits_{\ _{1}x}^{x}(x-\ x^{\prime })^{s-\alpha
-1}\left( \frac{\partial }{\partial x^{\prime }}\right) ^{s}f(x^{\prime
})dx^{\prime },  \label{lfcd} \\
 \mbox{and }
\ _{\ x}\overset{\alpha }{\underline{\partial }}_{\ _{2}x}f(x)&:=&\frac{1}{%
\Gamma (s-\alpha )} \int\limits_{x}^{\ _{2}x}(x^{\prime }-x)^{s-\alpha -1}
\left( -\frac{\partial }{\partial x^{\prime }}\right) ^{s}f(x^{\prime
})dx^{\prime }\ , \notag
\end{eqnarray}%
where we underline the partial derivative symbol, $\underline{\partial },$
in order to distinguish the Caputo operators from the RL ones with usual $%
\partial .$ A very important property is that for a constant $C,$ for
instance, $_{\ _{1}x}\overset{\alpha }{\underline{\partial }}_{x}C=0.$ In
our approach, we shall give priority to the fractional left Caputo
derivative resulting in constructions which are very similar
to those with integer calculus.

\subsection{Vector operations and non--integer differential forms}

\subsubsection{Fractional \ integral}

To formulate fractional integral (Gauss's, Stokes', Green's etc) theorems
with a formal noninteger order $\alpha $ integral $_{\ _{1}x}\overset{\alpha
}{I}_{\ _{2}x},$ we need a generalization of the Newton--Leibniz formula $%
_{\ _{1}x}\overset{\alpha }{I}_{\ _{2}x} \left(\ _{\ _{1}x}\overset{\alpha}{%
\partial }_{x} f(x)\right) =f(\ _{2}x)-f(\ _{1}x),$ for a fractional
derivative $\ _{\ _{1}x}\overset{\alpha}{\partial }_{x}.$ Such "mutually
inverse'' operations do not exist for an arbitrary taken type of fractional
derivative.

Let us denote by $L_{z}(\ _{1}x,\ _{2}x)$ the set of those Lebesgue
measurable functions $f$ on $[\ _{1}x,\ _{2}x]$ for which $||f||_{z}=(
\int\limits_{_{1}x}^{_{2}x}|f(x)|^{z}dx) ^{1/z}<\infty .$ We write $C^{z}[\
_{1}x,\ _{2}x]$ is a space of functions, which are $z$ times continuously
differentiable on this interval.

Using the fundamental theorem of fractional calculus \cite{taras1}, we have
that for any real--valued function $f(x)$ defined on a closed interval $[\
_{1}x,\ _{2}x],$ there is a function $F(x)=_{\ _{1}x}\overset{\alpha }{I}%
_{x}\ f(x)$ defined by the fractional Riemann--Liouville integral $\ _{\
_{1}x}\overset{\alpha }{I}_{x}f(x):=\frac{1}{\Gamma (\alpha )}%
\int\limits_{_{1}x}^{x}(x-x^{\prime })^{\alpha -1}f(x^{\prime })dx^{\prime
}, $ when the function $f(x)=\ _{\ _{1}x}\overset{\alpha }{\underline{%
\partial }}_{x}F(x)$ satisfies the conditions
\begin{eqnarray*}
\ _{\ _{1}x}\overset{\alpha }{\underline{\partial }}_{x}\left( _{\ _{1}x}%
\overset{\alpha }{I}_{x}f(x)\right) &=&f(x),\ \alpha >0, \\
_{\ _{1}x}\overset{\alpha }{I}_{x}\left( \ _{\ _{1}x}\overset{\alpha }{%
\underline{\partial }}_{x}F(x)\right) &=&F(x)-F(\ _{1}x),\ 0<\alpha <1,
\end{eqnarray*}%
for all $x\in \lbrack \ _{1}x,\ _{2}x].$ So, the right fractional RL
integral is inverse to the right fractional Caputo derivative.

There is a corresponding fractional generalization for the Taylor formula%
\begin{equation*}
f(\ _{2}x)-f(\ _{1}x)=_{\ _{1}x}\overset{\alpha }{I}_{x}\left( _{\ _{1}x}%
\overset{\alpha }{\underline{\partial }}_{x}f(x)\right)
+\sum\limits_{s_{1}=0}^{s-1}\frac{1}{s_{1}!}(\ _{2}x-\
_{1}x)^{s_{1}}f^{(s_{1})}(\ _{1}x),\
\end{equation*}%
for $s-1<\alpha \leq s,$ where $f^{(s_{1})}(x)=_{\ _{1}x}\overset{\alpha }{%
\underline{\partial }}_{x}f(x).$

\subsubsection{Definition of fractional vector operations}

Let $X$ be a domain of $\mathbb{R}^{n}$ parametrized as $X=\{\ _{1}x^{i}\leq
x^{i}\leq \ _{2}x^{i}\}$ (or, in brief, $X=\{\ _{1}x\leq x\leq \ _{2}x\},$
which substitute the closed one dimensional interval $[\ _{1}x,\ _{2}x].$
Let $f(x^{i})$ and $F_{k}(x^{i})$ be real--valued functions that have
continuous derivatives up to order $k-1$ on $X,$ such that the $k-1$
derivatives are absolutely continuous, i.e., $f,\mathbf{F=\{}F_{i}\mathbf{\}}%
\in AC^{k}[X],$ see details in \cite{kilbas}.

For a basis $e^{i}$ on $X,$ we can define a fractional generalization of
gradient operator $\partial =e^{i}\partial _{i}=e^{i}\partial /\partial
x^{i} ,$ when$\ _{\ _{1}x}\overset{\alpha }{\underline{\partial }}=e^{i}\
_{\ _{1}x}\overset{\alpha }{\underline{\partial }}_{i},$ for $\ _{\ _{1}x}%
\overset{\alpha }{\underline{\partial }}_{i}:=\ _{\ _{1}x}\overset{\alpha }{%
\underline{\partial }}_{x^{i}}$ being the left fractional Caputo derivatives
on $x^{i}$ defined by (\ref{lfcd}). If $f(x)=f(x^{i})$ is a $(k-1)$ times
continuously differentiable scalar field such that $\partial _{i}^{k-1}f$ is
absolutely continuous, we can define the fractional gradient of $\ f,$%
 $\overset{\alpha }{grad}\ f:=\ _{\ _{1}x}\overset{\alpha }{\underline{%
\partial }}f=e^{i}\ _{\ _{1}x}\overset{\alpha }{\underline{\partial }}_{i}f.$ 
 Let us consider that $X\subset \mathbb{R}^{n}$ \ has a flat metric $\eta
_{ij}$ and its inverse $\eta ^{ij}.$ Then we can define $F^{i}=\eta
^{ij}F_{j}(x^{k})$ and construct the fractional divergence operator%
 $\overset{\alpha }{div}\mathbf{F:}=_{\ _{1}x}\overset{\alpha }{\underline{%
\partial }}_{i}F^{i}(x^{k}).$ 

We can not use the Leibniz rule in a fractional generalization of the vector
calculus because for two analytic functions $\ ^{1}f$ and $\ ^{2}f$ we have
\begin{eqnarray*}
\overset{\alpha }{grad}(\ ^{1}f\ \ ^{2}f) &\neq & (\overset{\alpha }{grad}\
^{1}f) \ ^{2}f+ (\ \overset{\alpha }{grad}\ ^{2}f\ ) \ ^{1}f, \\
\overset{\alpha }{div}(f\mathbf{F)} &\neq & (\ \overset{\alpha }{grad}\ f\ ,%
\mathbf{F}) +f\ \ \overset{\alpha }{div}\mathbf{F.}
\end{eqnarray*}%
This follows from the property that $\ _{\ _{1}x^{\prime }}\overset{\alpha }{%
\underline{\partial }}_{i^{\prime }} (\ ^{1}f\ (x^{j^{\prime }}) \
^{2}f(x^{k^{\prime }})) \neq $\newline
$( _{\ _{1}x^{\prime }}\overset{\alpha }{\underline{\partial }}_{i^{\prime
}}\ ^{1}f(x^{j^{\prime }}))\ ^{2}f(x^{k^{\prime }})+(\ _{\ _{1}x^{\prime }}%
\overset{\alpha }{\underline{\partial }}_{i^{\prime }}\ ^{2}f(x^{k^{\prime
}}))\ ^{1}f(x^{j^{\prime }}).$

A fractional volume integral is a triple fractional integral within a region
$X\subset \mathbb{R}^{3},$ for instance, of a scalar field 
$f(x^{k}),$ $
\overset{\alpha }{I}(f)=\overset{\alpha }{I}[x^{k}]f(x^{k})=\overset{\alpha }%
{I}[x^{1}]\ \overset{\alpha }{I}[x^{2}]\ \overset{\alpha }{I}[x^{3}]f(x^{k}).$ 
 For $\alpha =1$ and $f(x,y,z),$ we have $\overset{\alpha }{I}%
(f)=\iiint\limits_{X}dV f=$ $\iiint\limits_{X}dxdydz\ f.$

The fundamental fractional integral theorems with $\overset{\alpha }{grad}$,
$\overset{\alpha }{div}$ etc are considered in \cite{taras1} for the
fractional Caputo derivatives. We omit such details in this work, but
emphasize that a self--consistent fractional integro--differential calculus on fractional manifolds,
which is very similar to the ''integer'' one, can be developed  following above constructions.

\subsubsection{Fractional differential forms}

There were elaborated different approaches to fractional differential form
caluculus. For instance, a fractional generalization of differential has
been presented by Ben Adda, see review of his results in \cite{benadda}.
Then different fractional generalizations of differential forms and
application to dynamical systems were proposed, see critical remarks and
original results in \cite{taras1,taras2} (fractional differentials/forms
were constructed using different fractional derivatives etc but fractional
integral theorems being considered only recently in \cite{taras1}).

Following \cite{taras1}, an exterior fractional differential is defined
through the fractional Caputo derivatives which is self--consistent with the
definition of fractional integral considered above. We introduce for $\
_{1}x^{i}=0$ the fractional absolute differential $\overset{\alpha }{d}$ \
in the form
\begin{equation*}
\overset{\alpha }{d}:=(dx^{j})^{\alpha }\ \ _{\ 0}\overset{\alpha }{%
\underline{\partial }}_{j}, \mbox{ where } \ \overset{\alpha }{d}%
x^{j}=(dx^{j})^{\alpha }\frac{(x^{j})^{1-\alpha }}{\Gamma (2-\alpha )}.
\end{equation*}%
For the ''integer'' calculus, we use as local coordinate co-bases/--frames
the differentials $dx^{j}=(dx^{j})^{\alpha =1}.$ The ''fractional'' symbol $%
(dx^{j})^{\alpha },$ related to $\overset{\alpha }{d}x^{j},$ is used instead
of $dx^{i}$ for elaborating a co--vector/differential form calculus, see
below the formula (\ref{frdif}). It is considered that for $0<\alpha <1$ we
have $dx=(dx)^{1-\alpha }(dx)^{\alpha }.$

An exterior fractional differential can be defined through the fractional
Caputo derivatives in the form%
\begin{equation*}
\overset{\alpha }{d}=\sum\limits_{j=1}^{n}\Gamma (2-\alpha )(x^{j})^{\alpha
-1}\ \overset{\alpha }{d}x^{j}\ \ _{\ 0}\overset{\alpha }{\underline{%
\partial }}_{j}.
\end{equation*}%
Differentials are dual to partial derivatives, and derivation is inverse to
integration. For a fractional calculus, \ the concept of ''dual'' and
''inverse'' have a more sophisticate relation to ''integration'' and, in
result, there is a more complex relation between forms and vectors.

The fractional integration for differential forms on $L=[\ _{1}x,\ _{2}x]$
is introduced using the operator %\begin{equation*}
$\ _{L}\overset{\alpha }{I}[x]:=\int\limits_{\ _{1}x}^{\ _{2}x}\frac{%
(dx)^{1-\alpha }}{\Gamma (\alpha )(\ _{2}x-x)^{1-\alpha }}$
%\end{equation*}%
when, for $0<\alpha <1,$%
\begin{equation}
\ _{L}\overset{\alpha }{I}[x]\ \ _{\ _{1}x}\overset{\alpha }{d}_{x}f(x)=f(\
_{2}x)-f(\ _{1}x).  \label{aux01}
\end{equation}%
The \textbf{fractional differential} of a function $\ f(x)$ is $_{\ _{1}x}%
\overset{\alpha }{d}_{x}f(x)=[...],$ with square brackets considered are
defined by formula%
\begin{equation*}
\int\limits_{_{1}x}^{_{2}x}\frac{(dx)^{1-\alpha }}{\Gamma (\alpha )(\
_{2}x-x)^{1-\alpha }}\left[ (dx^{\prime })^{\alpha }\ \ _{\ _{1}x}\overset{%
\alpha }{\underline{\partial }}_{x^{\prime \prime }}f(x^{\prime \prime })%
\right] =f(x)-f(\ _{1}x).
\end{equation*}

The exact fractional differential 0--form is a fractional differential of
the function
\begin{equation*}
\ _{\ _{1}x}\overset{\alpha }{d}_{x}f(x):=(dx)^{\alpha }\ \ _{\ _{1}x}%
\overset{\alpha }{\underline{\partial }}_{x^{\prime }}f(x^{\prime })
\end{equation*}%
when the equation (\ref{aux01}) is considered as the fractional
generalization of the integral for a differential 1--form. So, the \textbf{%
fractional exterior derivative} can be written
\begin{equation}
\ \ _{\ _{1}x}\overset{\alpha }{d}_{x}:=(dx^{i})^{\alpha }\ \ _{\ _{1}x}%
\overset{\alpha }{\underline{\partial }}_{i}.  \label{feder}
\end{equation}%
Then, the fractional differential 1--form $\ \overset{\alpha }{\omega }$
with coefficients $\{F_{i}(x^{k})\}$ is
\begin{equation}
\ \overset{\alpha }{\omega }=(dx^{i})^{\alpha }\ F_{i}(x^{k})  \label{fr1f}
\end{equation}

The exterior fractional derivatives of 1--form $\ \overset{\alpha }{\omega }$
gives a fractional 2--form,
\begin{equation*}
\ \ _{\ _{1}x}\overset{\alpha }{d}_{x}(\ \overset{\alpha }{\omega }%
)=(dx^{i})^{\alpha }\wedge (dx^{j})^{\alpha }\ _{\ _{1}x}\overset{\alpha }{%
\underline{\partial }}_{j}\ F_{i}(x^{k}).
\end{equation*}%
This rule can be proven following the relations \cite{kilbas} that for any
type fractional derivative $\ \overset{\alpha }{\partial }_{x\ },$ we have
\begin{equation*}
\ \overset{\alpha }{\partial }_{x\ }(\ ^{1}f\ \
^{2}f)=\sum\limits_{k=0}^{\infty }\left(
\begin{array}{c}
\alpha \\
k%
\end{array}%
\right) \left( \overset{\alpha -k}{\partial }_{x\ }\ ^{1}f\right) \ \
\overset{\alpha =k}{\partial }_{x\ }\left( \ ^{2}f\right) ,
\end{equation*}%
for integer $k,$ where
\begin{equation*}
\left(
\begin{array}{c}
\alpha \\
k%
\end{array}%
\right) =\frac{(-1)^{k-1}\alpha \Gamma (k-\alpha )}{\Gamma (1-\alpha )\Gamma
(k+1)}\mbox{\ and \ }\ \ \ \overset{k}{\partial }_{x\ }\ (dx)^{\alpha
}=0,k\geq 1.
\end{equation*}%
There are used also the properties:\ $\ _{_{1}x^{\prime }}\overset{\alpha }{%
\partial }_{x^{\prime }\ }(x^{\prime }-\ _{1}x^{\prime })^{\beta }=\frac{%
\Gamma (\beta +1)}{\Gamma (\beta +1-\alpha )}(x-\ _{1}x)^{\beta -\alpha },$
where $n-1<\alpha <n$ and $\beta >n,$ and $\ _{_{1}x^{\prime }}\overset{%
\alpha }{\partial }_{x^{\prime }\ }(x^{\prime }-\ _{1}x^{\prime })^{k}=0$
for $k=0,1,2,...,n-1.$ We obtain $\ _{\ _{1}x}\overset{\alpha }{d}_{x}(x-\
_{1}x)^{\alpha }=(dx)^{\alpha }\ _{\ _{1}x}\overset{\alpha }{\underline{%
\partial }}_{i^{\prime }}x^{i^{\prime }}=(dx)^{\alpha }\Gamma (\alpha +1),$
i.e.
\begin{equation}
(dx)^{\alpha }=\frac{1}{\Gamma (\alpha +1)}\ _{\ _{1}x}\overset{\alpha }{d}%
_{x}(x-\ _{1}x)^{\alpha }\   \label{frdif}
\end{equation}%
and write the fractional exterior derivative (\ref{feder}) in the form
\begin{equation*}
\ \ _{\ _{1}x}\overset{\alpha }{d}_{x}:=\frac{1}{\Gamma (\alpha +1)}\ \ _{\
_{1}x}\overset{\alpha }{d}_{x}(x^{i}-\ _{1}x^{i})^{\alpha }\ _{\ _{1}x}%
\overset{\alpha }{\underline{\partial }}_{i}.
\end{equation*}%
Using this formula, the \ fractional differential 1--form (\ref{fr1f}) can
be written alternatively%
\begin{equation*}
\ \overset{\alpha }{\omega }=\frac{1}{\Gamma (\alpha +1)}\ \ _{\ _{1}x}%
\overset{\alpha }{d}_{x}(x^{i}-\ _{1}x^{i})^{\alpha }\ \ F_{i}(x).
\end{equation*}

Having a well defined exterior calculus of fractional differential forms on
flat spaces $\mathbb{R}^{n},$ we can generalize the constructions for a real
manifold \ $M,\dim M=n.$\

\end{document}